\documentclass[a4paper]{amsart}
\usepackage{tikz}
\usetikzlibrary{calc, decorations.pathreplacing, positioning}
\usepackage[english]{babel}
\usepackage[T1]{fontenc}
\usepackage{lmodern}
\usepackage[centering]{geometry}

\usepackage{MnSymbol}

\theoremstyle{plain}
\newtheorem{theo}{Theorem}

\newtheorem{prop}[theo]{Proposition}
\newtheorem{coro}[theo]{Corollary}

\theoremstyle{definition}
\newtheorem{defi}[theo]{Definition}

\theoremstyle{remark}
\newtheorem{rema}[theo]{Remark}

\usepackage{microtype}
\usepackage{enumerate}
\usepackage{graphicx}
\usepackage{color}

\usepackage{bm}
\usepackage{mathtools}
\usepackage[dvipsnames]{xcolor}
\definecolor{FlatRed}{RGB}{231,76,60}
\definecolor{FlatGreen}{RGB}{46,204,113}
\definecolor{FlatBlue}{RGB}{52,152,219}
\definecolor{FlatYellow}{RGB}{241,196,15}
\colorlet{FlatViolet}{FlatRed!50!FlatBlue}
\colorlet{FlatBrown}{FlatRed!50!FlatGreen}
\colorlet{FlatOrange}{FlatRed!50!FlatYellow}
\colorlet{FlatCyan}{FlatGreen!50!FlatBlue}

\usepackage{enumitem}
\usepackage{textcomp}
\usepackage[normalem]{ulem}

\usepackage{hyperref}
\hypersetup{
    colorlinks,
    linkcolor={FlatRed},
    citecolor={FlatGreen},
    urlcolor={FlatBlue}
}

\usepackage{nicematrix}
\usepackage{upgreek}

\title[Scaling Limit of the Hierarchical $H^{2|2}$ model]{Fine-Graining and Continuous Space Scaling Limit of the $H^{2|2}$ Model on the Hierarchical Lattice}

\author{Yichao Huang}
\address{Beijing Institute of Technology, School of Mathematics and Statistics, China}
\email{yichao.huang@bit.edu.cn}

\author{Jinglin Wang}
\address{Institut de Recherche Mathématique Avancée, Université de Strasbourg, France}
\email{jinglin.wang@math.unistra.fr}

\author{Xiaolin Zeng}
\address{Institut de Recherche Mathématique Avancée, Université de Strasbourg, France}
\email{zeng@math.unistra.fr}

\begin{document}

\begin{abstract}
We extend the exact coarse-graining result of Disertori, Merkl and Rolles~\cite{MR4517733} for the random field of $H^{2|2}$-model to the random Schrödinger operator representation of the $H^{2|2}$-model. We also introduce a fine-graining procedure as the reverse operation, and establish an associated exponential martingale property. Applying this fine-graining procedure to the $H^{2|2}$-model on the Dyson hierarchical lattice, we establish its continuous space scaling limit as a non-trivial random measure on $[0,1]$.

This random measure is almost surely singular with respect to the Lebesgue measure if and only if the Vertex Reinforced Jump Process on the Dyson hierarchical lattice is recurrent. If the process is transient, the random measure almost surely has an absolutely continuous component. The density of this component is everywhere non-trivial and can be identified with the pointwise limit of an exponential martingale associated with the $H^{2|2}$-model on the Dyson hierarchical lattice.
\end{abstract}

\maketitle

\section{Introduction}
The study of disordered quantum systems, particularly the phenomenon of Anderson localization, has greatly benefited from non-perturbative field-theoretic methods. A cornerstone of this approach is the Wegner-Efetov supersymmetry formalism, which maps problems concerning Random Schrödinger Operators to the study of a specific type of sigma model. Within this framework, the supersymmetric hyperbolic sigma model, or the $H^{2|2}$-model, has emerged as a paradigmatic theory~\cite{MR1178886} whose effective degrees of freedom are believed to capture the universal properties of localization transitions.

A remarkable bridge, first established by Sabot and Tarrès~\cite{MR3420510}, connects this intriguing theory of mathematical physics to the realm of self-interacting stochastic processes. This work relates the phase transition of the field theory to the long-term behavior of the Vertex Reinforced Jump Process (VRJP), a non-Markovian process with long-range memory effects. The ``effective field'' of the $H^{2|2}$-model serves as the random environment that governs the VRJP. This correspondence is profound, translating physical questions about disorder and criticality into probabilistic questions of recurrence and transience for a self-interacting random process.

In this work, we investigate this correspondence in a setting of long-range interactions on the {\it Dyson hierarchical lattice}. This graph has vertex set $\mathbb{X} = \{1, 2, \dots\}$ and is endowed with the ultrametric distance defined as the height of the smallest subtree containing both $i$ and $j$ (see Figure \ref{fig1}). Formally:
\begin{equation*}
    d_{\mathbb{X}}(i,j) = \min\{ n \ge 0 \mid \exists k \in \mathbb{N} \text{ s.t. } i, j \in (k 2^n, (k+1)2^n] \}.
\end{equation*}
The two-body interaction decays exponentially with this distance: $W_{ij} = \overline{W} (2\rho)^{-d_{\mathbb{X}}(i,j)}$, where $\overline{W} > 0$ is the inverse temperature and $\rho > 1$. The hierarchical structure serves as a tractable toy model for systems on $\mathbb{Z}^d$, where the parameter $\rho$ plays a role analogous to the spatial dimension. Our central contribution is the construction of a continuum limit for the model's effective field and a characterization of its properties via long-term behavior of the VRJP.

\begin{figure}[h]
    \centering
    \scalebox{0.84}{
\begin{tikzpicture}[
    xscale=1.5, yscale=1.2,
    site/.style={circle, fill=black, inner sep=2pt, label=below:{\footnotesize $#1$}},
    link/.style={thick, gray!30}, 
    active/.style={ultra thick, red!80!black}, 
    interaction/.style={<->, dashed, red!80!black, thin},
    axis_label/.style={font=\footnotesize, fill=white, inner sep=1pt, text=black!80}
]

    \def\numlevels{3}
    \pgfmathtruncatemacro{\numsites}{2^\numlevels}
    
    \draw[->, thick, black!70] (0.5, 0) -- (0.5, \numlevels+0.5) node[above, black] {$d_{\mathbb{X}}(i,j)$};
    
    \foreach \y in {1,...,\numlevels} {
        \draw[dotted, gray!50] (0.5, \y) -- (\numsites+0.5, \y);
        \node[left, font=\scriptsize, gray] at (0.5, \y) {$n=\y$};
    }

    \foreach \i in {1,...,\numsites} {
        \node[site=\i] (s-\i) at (\i, 0) {};
    }

    \foreach \level in {1,...,\numlevels} {
        \pgfmathtruncatemacro{\prevlevel}{\level-1}
        \pgfmathtruncatemacro{\numblocks}{2^(\numlevels-\level)}
        
        \foreach \k in {1,...,\numblocks} {
            \pgfmathtruncatemacro{\idxL}{2*\k-1}
            \pgfmathtruncatemacro{\idxR}{2*\k}
            
            \ifnum\level=1
                \coordinate (childL) at (s-\idxL);
                \coordinate (childR) at (s-\idxR);
            \else
                \coordinate (childL) at (p-\prevlevel-\idxL);
                \coordinate (childR) at (p-\prevlevel-\idxR);
            \fi
            
            \coordinate (parent) at ($(childL)!0.5!(childR) + (0, 1)$);
            \coordinate (p-\level-\k) at (parent);
            
            \draw[link] (childL) -- (childL |- parent) -- (childR |- parent) -- (childR);
            
            \node[circle, fill=white, draw=gray!40, inner sep=1.5pt, scale=0.6] at (parent) {};
        }
    }

    
    \draw[active] (s-2) -- (s-2 |- p-1-1) -- (p-1-1);
    \draw[active] (p-1-1) -- (p-1-1 |- p-2-1) -- (p-2-1);
    \draw[active] (p-2-1) -- (p-2-1 |- p-3-1) -- (p-3-1);
    
    \draw[active] (p-3-1) -- (p-2-2 |- p-3-1) -- (p-2-2);
    \draw[active] (p-2-2) -- (p-1-4 |- p-2-2) -- (p-1-4);
    \draw[active] (p-1-4) -- (s-7 |- p-1-4) -- (s-7);

    \node[circle, fill=red, inner sep=2pt] at (p-3-1) {};
    \node[right, font=\scriptsize, red!80!black] at (p-3-1) {};

    
    \draw[interaction, bend right=30] (s-2) to node[midway, below=2pt, fill=white, inner sep=1pt] 
        {\scriptsize $W_{2,7} = \overline{W}  (2\rho)^{-3}$} (s-7);

    \node[anchor=north west, align=left, font=\small] at (1, \numlevels+1.5) {
        Distance: $d_{\mathbb{X}}(2,7) = 3$\\
        Interaction mediated by block $B_3$.
    };

\end{tikzpicture}

 }
\caption{Hierarchical distance and couping.}
\label{fig1}
\end{figure}
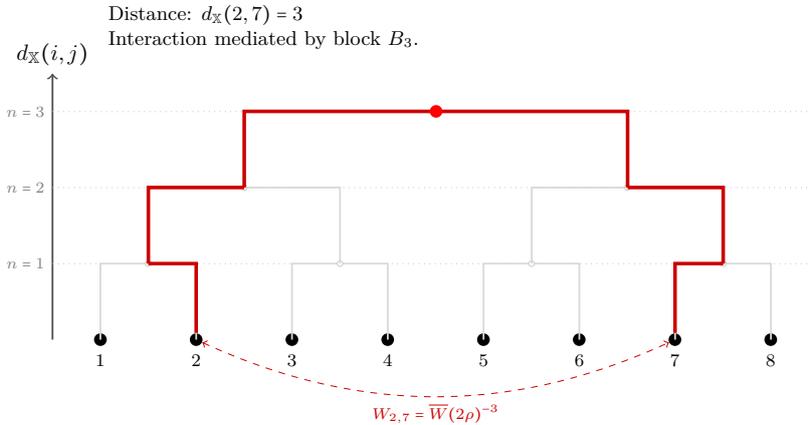

To achieve this, we develop a novel ``fine-graining'' renormalization scheme. We consider the sequence of finite subgraphs $\Lambda_n = \{1, \dots, 2^n\}$ and embed them into the unit interval $[0,1]$ by identifying each vertex $k \in \Lambda_n$ with the dyadic interval $[\frac{k-1}{2^n}, \frac{k}{2^n}]$. The fine-graining splits each $k\in \Lambda_{n}$ into $\{2k-1,2k\}\in \Lambda_{n+1}$. In this setup, we prove the almost sure existence of a limiting object, which we call the {\it continuum space limit} of the model. This limit takes the form of a random measure $\mathfrak{m}_{\infty}$ on $[0,1]$. Our main result establishes a direct link between the behavior of the associated VRJP and the analytic properties of this measure.

\medskip

\noindent{\bf  Theorem (Informal).} Consider the $H^{2|2}$-model on the hierarchical lattice $\mathbb{X}$ at the critical dimension $2$. For any inverse temperature $\overline{W}$, a suitable continuous space scaling limit can be constructed as a limiting random measure $\mathfrak{m}_{\infty}$ on $[0,1]$ and the associated VRJP are related to it as follows:
\begin{itemize}
\item If the VRJP on $\mathbb{X}$ is recurrent, then the measure $\mathfrak{m}_{\infty}$ is almost surely singular with respect to the Lebesgue measure.
\item If the VRJP on $\mathbb{X}$ is transient, then the measure $\mathfrak{m}_{\infty}$ almost surely possesses a non-trivial absolutely continuous component with strictly positive density function eveywhere on $[0,1]$.
\end{itemize}

\medskip

This theorem provides a rigorous connection between the long time behavior of the VRJP and the spectral decomposition of the limiting random measure. We conjecture that for the critical dimension $2$, the VRJP on the Dyson hierarchical lattice is in fact recurrent for all values of the inverse temperature $\overline{W}$. If this conjecture holds, our result would imply that the continuum limit of the $H^{2|2}$-model at the critical dimension is always described by a singular measure. 

Our key technical innovation is the fine-graining procedure, which can be understood as an inverse of the more traditional coarse-graining renormalization. This method applies to any graph having a subset that is indistinguishable by its complementary, reverses the coarse-graining results of~\cite{MR4517733}, and provides a new tool for constructing continuum limits of statistical mechanical models on hierarchical structures (with a real-space renormalization). In particular, a key martingale property extending the martingale property discovered in~\cite[Theorem 1.(ii)]{MR3904155} is established to ensure the almost sure scaling limit.

The remainder of this paper is organized as follows. Section~\ref{sec:coarse_and_fine_graining_properties_of_the_h_2_2} recalls the definition of the $H^{2|2}$-model, extends the coarse-graining property of the effective field to a related random potential, and develops the reverse fine-graining procedure and the associated one-step martingale property, all on general graphs. Section~\ref{sec:the_h_2_2} applies these results to the $H^{2|2}$-model on the hierarchical lattice, introducing a new fine-graining coupling suitable for renormalization in the ultra-violet limit. Finally, Section~\ref{sec:continuous_space_scaling_limit_of_the_hiarachical_h_2_2} is devoted to proving the existence of the scaling limit and relates its Lebesgue decomposition to the long-time behavior of the VRJP.

\subsection*{Acknowledgements.} Y.H. is partially supported by National Key R\&D Program of China (No. 2022YFA1006300) and NSFC-12301164. J.W. and X.Z. acknowledge the support of the Institut de recherche en mathématiques, interactions \& applications: IRMIA++.

\section{Coarse and fine-graining properties of the $H^{2|2}$-model}\label{sec:coarse_and_fine_graining_properties_of_the_h_2_2}
In this section, we provide some backgrounds on the $H^{2|2}$-model on general graphs, then discuss its coarse and fine-graining properties. We also derive in Theorem~\ref{th:ExponentialMartingale} a one-step martingale property, which is a crucial input for the rest of the article when we construct the continuum limit.

\subsection{Definition of the $H^{2|2}$-model via its random Schrödinger operator representation}\label{subse:definition_of_the_h_2_2}
Let $\mathcal{G}=(V,E,W)$ be a (simple unoriented) finite connected weighted graph, with $V$ the set of its vertices, $E$ the set of its edges, with each edge denoted by $e=\{i,j\}$ for a pair of vertices $(i,j)\in V$, and $W$ the set of edge weights defined to be positive real numbers $W_e=W_{ij}=W_{ji}>0$ for each edge $e\in E$. We also set $W_{ii}=0$ for all $i\in V$, with the convention that $W_{ij}=0$ is equivalent to the absence of edges between vertices $i,j$.

Recall that the weighted graph Laplacian $\Delta_W$ is the self-adjoint operator on $l^{2}(V)$ such that for all $f:V\to\mathbb{R}$,
\begin{equation*}
    \forall i\in V,\quad (-\Delta_W f)(i)=(\sum\limits_{j:j\sim i}W_{ij})f(i)-\sum\limits_{j:j\sim i}W_{ij}f(j)
\end{equation*}
where $j\sim i$ indicates that $j$ is a neighbor of $i$. Following~\cite{MR3904155}, we define the $H^{2|2}$-model on the graph $\mathcal{G}$ by considering a random Schrödinger operator $H_{\beta}$, which acts on $f\in l^{2}(V)$ as
\begin{equation*}
    \forall i\in V,\quad (H_{\beta} f)(i)=2\beta_if(i)-\sum\limits_{j:j\sim i}W_{ij}f(j).
\end{equation*}
It has the following matrix representation in the canonical basis $(e_i)_{i\in V}$ of $l^{2}(V)$:
\begin{equation}\label{eq:Form_H}
H_{\beta}=\begin{pmatrix}
2\beta_i & \cdots & -W_{ij} \\
 \vdots & \ddots & \vdots \\
-W_{ij} & \cdots & 2\beta_j \\
\end{pmatrix}.
\end{equation}
In particular, its diagonal entries are determined by a vector $\beta\in \mathbb{R}^{V}$, which we call the random potential in \(H^{2|2}\)-model.

The law of the random vector $(\beta_{i})_{i\in V}$ for the $H^{2|2}$-model is given explicitly by the following probability distribution:
\begin{equation}\label{eq:Distribution_Beta}
    d\nu^{W}(\beta)=\left(\sqrt{\frac{2}{\pi}}\right)^{|V|}\mathbf{1}_{\{H_\beta>0\}}e^{-\frac{1}{2}\langle1,H_\beta 1\rangle} \frac{1}{\sqrt{\det H_\beta}}\prod_{i\in V}d\beta_{i}.
\end{equation}
where the inner product $\langle x,y\rangle=\sum_{i\in V}x_iy_i$ is the canonical one and in $\left<1,H_{\beta} 1\right>$, $1$ abusively denotes the constant vector with all entries equal to $1$. The indicator $\{H_\beta>0\}$ restricts the distribution to positive definite matrices $H_\beta$. It is a non-trivial fact that $d\nu^{W}(\beta)$ is a probability measure~\cite[Theorem~1]{MR3729620}, and that $\beta_i$ are $\nu^{W}$-almost surely positive for all $i\in V$. In this way, $H_{\beta}$ corresponds to a random Schrödinger operator with a random potential which is not i.i.d. but correlated according to~\eqref{eq:Distribution_Beta}. The measure $d\nu^{W}(\beta)$ generalizes the Gamma$(\frac{1}{2})$ distrubtion in the random Schrödinger operator setting.

The distribution $d\nu^{W}$ above belongs to a larger family of probability measures with extra parameters $\theta\in\mathbb{R}^{V}_{>0}$~\cite{MR3729620} and $\eta\in\mathbb{R}^{V}_{+}$~\cite{letac2017multivariate} defined as
\begin{equation}\label{eq:MultiInverseGaussian}
    d\nu^{W,\theta,\eta}(\beta)=\left(\sqrt{\frac{2}{\pi}}\right)^{|V|}\mathbf{1}_{\{H_\beta>0\}}e^{-\frac{1}{2}\left(\langle\theta,H_\beta\theta\rangle+\langle\eta,H_\beta^{-1}\eta\rangle-2\langle\theta,\eta\rangle\right)}\frac{\prod_{i\in V}\theta_i}{\sqrt{\det H_\beta}}\prod_{i\in V}d\beta_{i},
\end{equation}
called multivariate (or random Schrödinger operator) generalization of the (reciprocal) inverse Gaussian distribution. In the most general setting with parameters $W,\theta,\eta$ as above, this law is characterized by its explicit Laplace transform:
\begin{equation}\label{eq:LaplaceTransform_MultiInverseGaussian}
\begin{split}
    \forall \lambda\in(\mathbb{R}_+)^{V},\quad &\int e^{-\langle\lambda,\beta\rangle} d\nu^{W,\theta,\eta}(\beta)\\
    ={}&e^{\sum\limits_{i\sim j}W_{ij}(\sqrt{(\lambda_{i}+\theta_{i}^2)(\lambda_{j}+\theta_{j}^2)}-\theta_{i}\theta_{j})-\sum\limits_{i\in V}\eta_{i}(\sqrt{\lambda_{i}+\theta_{i}^2}-\theta_{i})}\prod_{i\in V}\frac{\theta_i}{\sqrt{\lambda_{i}+\theta_{i}^2}}.
\end{split}
\end{equation}
We call the random Schrödinger operator whose diagonal entries are sampled with the law~\eqref{eq:Distribution_Beta} the $H^{2|2}$-model. More generally, when the entries are sampled from the law~\eqref{eq:MultiInverseGaussian}, we refer to the operator as the $H^{2|2}$-model with initial local time $\theta$ and boundary condition $\eta$. 

In fact, as an immediate consequence of \eqref{eq:LaplaceTransform_MultiInverseGaussian}, distribution defined in \eqref{eq:MultiInverseGaussian} is related to the one in \eqref{eq:Distribution_Beta} in the following way:  given \( \widetilde{\mathcal{G}}=(\widetilde{V},\widetilde{E},\widetilde{W}) \) with $\widetilde{V}=V\cup\{\delta\}$, and parameters $\widetilde{\theta}$, we have
\begin{equation}\label{eq:nu-eta-is-marginal}\begin{aligned}
 (\beta_{i})_{i\in \widetilde{V}} \sim \nu^{\widetilde{W},\widetilde{\theta},0} \implies (\beta_{i})_{i\in V} \sim \nu^{W,\theta,\eta},
\end{aligned}\end{equation}
where \(\eta_{i}=\widetilde{W}_{\delta,i}\), $\theta_{i}=\widetilde{\theta}_{i}$ for all $i\in V$ and $W_{ij}=\widetilde{W}_{ij}$ for all $i,j\in V$. In this paper, $\theta$ is always constant 1.

In what follows, we drop the subscript $\beta$ when there is no ambiguity, in particular, $H=H_{\beta}$.

\subsection{Green function and random walk expansion of the $H^{2|2}$-model}\label{subse:green_function_and_random_walk_expansion_of_the_h_2_2}
It is useful to introduce the matrix inverse $G=H^{-1}$ of $H$, i.e. the random Green function matrix of the $H^{2|2}$-model on the graph $\mathcal{G}$. Since the random Schrödinger matrix $H$ is almost surely positive definite (in particular, $H$ is almost surely an $M$-matrix~\cite[Chapter~6]{MR1298430}), we have the following (deterministic) random walk expansion representation~\cite{MR719815} of $G(i,j)$,  according to~\cite[Proposition~6]{MR3904155}:
\begin{equation*}
    \forall i,j\in V,\quad G(i,j)=\sum_{\sigma:i\to j}\frac{W_{\sigma}}{(2\beta)_\sigma}
\end{equation*}
where we sum over finite length nearest neighbor random walk paths $\sigma:i\to j$. If the path $\sigma$ has length $|\sigma|=n$, so that $\sigma=(\sigma_{k})_{k=0,\dots,n}$ with $\sigma_0=i$ and $\sigma_n=j$, then the quantities $W_{\sigma}$ and $(2\beta)_\sigma$ are defined as
\begin{equation*}
    W_{\sigma}=\prod_{k=0}^{n-1}W_{\sigma_k,\sigma_{k+1}},\quad (2\beta)_{\sigma}=\prod_{k=0}^{n}2\beta_{\sigma_k}.
\end{equation*}

\begin{rema}\label{rema:RandomWalkGreen}
Since the map $H\mapsto G=H^{-1}$ is a bijection, when the random walk expansion of $G(i,j)$ holds for all $i,j\in V$, $G$ is necessarily the inverse of a matrix of the form~\eqref{eq:Form_H}, as a consequence of~\cite[Proposition~5.3]{MR4254805}.
\end{rema}

\subsection{Coupling with different pinning positions}\label{subse:coupling_with_different_pinning_positions}
The random Schrödinger operator representation of the $H^{2|2}$-model on $\mathcal{G}$ provides a coupling of the so-called effective measures of the $H^{2|2}$-model on $\mathcal{G}$ with different pinning positions. We give a brief description of this fact following~\cite[Lemma~2]{MR3904155}.

To this end, define the random $u$-field associated to the $H^{2|2}$-model on $\mathcal{G}$. Fix $i_0\in V$ and consider the random vector $(u_i)_{i\in V}$ by setting
\begin{equation*}
    \forall i\in V,\quad e^{u_i}=\frac{G(i_0,i)}{G(i_0,i_0)}.
\end{equation*}
This new system of coordinates depends on the choice of pinning point $i_0$: we should have denoted them by $e^{u^{(i_{0})}_{i}}= \frac{G(i_{0},i)}{G(i_{0},i_{0})}$, in particular, $u^{(i_{0})}_{i_{0}}=u_{i_0}=0$ at the pinning vertex $i_0$. We drop the superscript $(i_{0})$ for notation ease and the pinning point will be clear in the context.

It can be shown~\cite[Theorem~2]{MR3420510} that the random $u$-field has the explicit probability law
\begin{equation}\label{eq:Distribution_u}
    d\mu_{i_0}^{W}(u)=\frac{1}{\sqrt{2\pi}^{|V|-1}}e^{-\sum\limits_{i\in V}u_i-\sum\limits_{i\sim j}W_{ij}(\cosh(u_i-u_j)-1)}\sqrt{D(W,u)}\mathbf{1}_{\{u_{i_0}=0\}}du
\end{equation}
where $D(W,u)$ is a sum over the spanning trees $T$ of the graph $\mathcal{G}$,
\begin{equation}\label{eq:DeterminantMatrixTree}
    D(W,u)=\sum_{T}\prod_{(ij)\in E(T)}W_{ij}e^{u_i+u_j}.
\end{equation}
Choosing a different $i_0$ yields a different measure $\mu^{W}_{i_0}$ with a different pinning condition.

It is shown in~\cite[Theorem~3]{MR3729620} that the random field $(u_i)_{i\in V}$ sampled from the above distribution $\mu_{i_0}^{W}$ is independent of $G(i_0,i_0)$ (and this latter distributed as an inverse Gamma$(\frac{1}{2})$ random variable), and that given $(u_i)_{i\in V}$ and $G(i_0,i_0)$, one reconstructs by linear algebra~\cite[Proposition~2]{MR3729620} the random $\beta$-field $(\beta_i)_{i\in V}$ by the following formula:
\begin{equation}\label{eq:RelationBeta_u}
    \forall i\in V,\quad 2\beta_i=\sum_{i\sim j}W_{ij}e^{u_j-u_i}+\mathbf{1}_{\{i=i_0\}}\frac{1}{G(i_0,i_0)}.
\end{equation}
In particular, the random field $(u_i)_{i\in V}$ is measurable with respect to the random field $(\beta_{i})_{i\in V\setminus\{i_0\}}$; conversely, at site $i_0$ one only needs one independent inverse Gamma$(\frac{1}{2})$ random variable to reconstruct $\beta_{i_0}$ from the data $(u_i)_{i\in V}$.

To summarize, given a finite graph $\mathcal{G}=(V,E,W)$, the $H^{2|2}$-model on $\mathcal{G}$ can refer to either of the following sets of information:
\begin{enumerate}
    \item The random vector $(\beta_i)_{i\in V}$ with law $\nu^{W}(d\beta)$ or the random Schrödinger operator $H$;
    \item The random Green function matrix $G=H^{-1}$;
    \item A fixed vertex $i_0\in V$, the random vector $(u_i)_{i\in V}$ with law $\mu^{W}_{i_0}$ and an independent inverse Gamma$(\frac{1}{2})$ random variable representing the variable $G(i_0,i_0)$.
\end{enumerate}
In the sequel, we will switch freely between these representations of the $H^{2|2}$-model.

The law $\mu_{i_0}^W$ was first introduced in~\cite{MR2728731} as the effective field of the supersymmetric hyperbolic sigma model ($H^{2|2}$-model). It was a surprising discovery~\cite{MR3420510} that this is also the law of the random environment for the Vertex Reinforced Jump Process (VRJP).

\subsection{Coarse-graining procedure}\label{subse:coarse_graining_procedure}
We now recall a coarse-graining property of the $H^{2|2}$-model, which was first observed in~\cite{MR4517733} in the case of the $H^{2|2}$-model on complete graphs with hierarchical interactions. Our formulation below differs slightly: we do not talk about the supersymmetric spin values, but we discuss additionally the random potential in the Schrödinger representation of \(H^{2|2}\)-model . We provide an alternative proof based on the random walk expansion recalled in Section~\ref{subse:green_function_and_random_walk_expansion_of_the_h_2_2}.

\begin{defi}[Indistinguishability]
Given a finite weighted graph $\mathcal{G}=(V,E,W)$, we say that a subset $U\subset V$ is \emph{indistinguishable from outside} if
\begin{equation*}
    \forall i,j\in U, \forall k\notin U,\quad W_{ki}=W_{kj}.
\end{equation*}
If additionally a boundary condition $\eta\in \mathbb{R}_{+}^{V}$ is given, $U$ is indistinguishable from outside if additionally
$$\forall i,j\in U,\ \eta_{i}=\eta_{j}.$$
In words, for a given vertex $k$ not in $U$, it is connected to all vertices in $U$ with the same weight. 
\end{defi}

When a set $U$ is indistinguishable from outside, we can coarse-grain this set into a single vertex using the following transformation:
\begin{theo}\label{th:CoarseGraining}
Consider the $H^{2|2}$-model (with boundary condition $\eta$) on a finite graph $\mathcal{G}$ by the data of its random Green function operator $G=H^{-1}$. If $U\subset V$ is a subset of vertices indistinguishable from outside, then one can effectively replace $U$ by a single vertex $u$ and coarse-grain the matrix $G$ into a matrix $G'$ on $l^{2}(V')$ with $V'=(V\setminus U)\cup\{u\}$, whose coefficients are defined by
\begin{equation}\label{eq:G'_G}
    G'(k,\ell)=\begin{cases} G(k,\ell) & k,\ell\notin U \\ \frac{1}{|U|}\sum_{i\in U}G(k,i) & k\notin U,\ell=u \\ \frac{1}{|U|^2}\sum_{i,j\in U}G(i,j) & k=\ell=u \end{cases}.
\end{equation}

The resulting matrix $G'$ has the same law of the Green operator defined by the $H^{2|2}$-model (with boundary condition $\eta'$) on the reduced graph $\mathcal{G}'=(V',E',W')$ with edge weights
\begin{equation}\label{eq:ReducedEdgeWeights}
    W'_{k,\ell}=\begin{cases}W_{k,\ell} & k,\ell\ne u \\ \sum_{i\in U}W_{k,i} & \ell=u, k\ne u\end{cases}.
\end{equation}
(and boundry condition is $\eta'_{u}= \sum_{i\in U}\eta_{i}$ and $\eta'_{k} = \eta_{k}$ for $k\ne u$).

Furthermore, the inverse of $G'$ is of the form~\eqref{eq:Form_H}, i.e. an $H^{2|2}$-model random Schrödinger operator on the reduced graph $\mathcal{G'}$, where the random vector $\beta'$ is defined by $\beta'_k=\beta_k$ if $k\in V'\setminus\{u\}$ and
\begin{equation}\label{eq:Definition_Beta'u}
    \frac{1}{2\beta'_u}=\frac{1}{|U|^2}\sum\limits_{i,j\in U}\widehat{G}^{U}(i,j)
\end{equation}
where $\widehat{G}^{U}=(H_{U,U})^{-1}$ with $H_{U,U}$ is the sub-matrix of $H$ on $U\times U$.
\end{theo}

There is a slight notation collision of the coarse-grained vertex $u$ and the $u$-field of \(H^{2|2}\)-model, but contextually it will always be clear.

\begin{rema}\label{rema:AlmostSure+Simultaneous}
\begin{enumerate}
\item Using \cite[Corollary 1]{MR3729620}, Theorem \ref{th:CoarseGraining} (and Theorem \ref{th:FineGrainBeta} below) can be formulated with general initial local time $\theta$. 
\item It is important to note that the above coarse-graining procedure is done in the almost sure sense simultaneously in $G$ and in the $\beta$-field. In particular, the last display~\eqref{eq:Definition_Beta'u} above shows that $\beta'_u$ can be written as a explicit function of $(\beta_i)_{i\in U}$. This formula shares similar flavor with the exact renormalization procedures of the classical random Schrödinger operators on the hierarchical lattice, see e.g.~\cite{MR1063180,MR3649447}.
\end{enumerate}
\end{rema}

\begin{proof}
By \eqref{eq:nu-eta-is-marginal} it is enough to consider the case without boundary condition. We first show that the inverse of $G'$ is indeed of the form~\eqref{eq:Form_H} with the above-defined random potential $\beta'$, then we show that the joint law of $(\beta'_j)_{k\in V'}$ is the correct one.

For the first step, we use Remark~\ref{rema:RandomWalkGreen}. We start by writing the random walk expansion of the operator $G$ defined on the original graph $\mathcal{G}$ for $k,l\notin U$ as in Section~\ref{subse:green_function_and_random_walk_expansion_of_the_h_2_2},
\begin{equation*}
    G(k,l)=\sum_{\sigma:k\to l}\frac{W_{\sigma}}{(2\beta)_\sigma}.
\end{equation*}
Decompose each finite path $\sigma$ from $k$ to $l$ in the following way: if $|\sigma|=m+1$ is the length of the path $\sigma$, let $k_0=k$ and $k_{m+1}=l$ and write
\begin{equation*}
    \sigma=\{k_0\to k_1\curvearrowright i_1\rightsquigarrow j_1\curvearrowright k'_1 \to k_2\curvearrowright i_2\rightsquigarrow j_2\curvearrowright k'_2~\cdots~\curvearrowright i_m\rightsquigarrow j_m\curvearrowright k'_m\to k_{m+1}\}
\end{equation*}
where the $k$ indexed vertices are outside $U$, the $i,j$ indexed vertices are inside $U$, $k'\to k$ is the sum over paths from $k'$ to $k$ without using the vertices in $U$, $k\curvearrowright i$ or $j\curvearrowright k$ means a single step jump that either enters or goes out of $U$, and $i\rightsquigarrow j$ is the sum over paths from $i$ to $j$ in $U$. In words, we record the instances when the path $\sigma$ enters or quits the set $U$. We say that $(k_q,i_q,j_q,k'_q)_{1\leq q\leq m}$ is the skeleton of $\sigma$, denoted as $\text{sk}(\sigma)$.

Consider the reduced path $\sigma'$ on the reduced graph $V'$ derived from the original path $\sigma$ on $V$ by tracing out jumps inside $U$, i.e.
\begin{equation*}
    \sigma'=\{k_0\to k_1\curvearrowright u\curvearrowright k'_1\to k_2\curvearrowright u\curvearrowright k'_2~\cdots~\curvearrowright u\curvearrowright k'_m\to k_{m+1}\},
\end{equation*}
and we call $(k_q,k'_q)_{1\leq q\leq m}$ the reduced skeleton of $\sigma$, denoted $\text{resk}(\sigma)$. Our goal is to write a reduced expansion of $G'(k,l)=G(k,l)$ in terms of a sum over the reduced paths $\sigma'$ on the reduced graph $V'$ with suitable choices of $W'$ and $\beta'$ (which will be shown to be given respectively by~\eqref{eq:ReducedEdgeWeights} and~\eqref{eq:Definition_Beta'u}):
\begin{equation*}
    G(k,l)=\sum\limits_{\sigma':k\to l}\frac{W'_{\sigma'}}{(2\beta')_{\sigma'}}.
\end{equation*}
If this can be done, then by posing $G'(k,l)=G(k,l)$ we get a random walk expansion of the reduced random operator $G'$. If such rewriting can be done for $G'(k,u)$ with $k\in V'\setminus\{u\}$ and $G'(u,u)$ defined in~\eqref{eq:G'_G}, then by Remark~\ref{rema:RandomWalkGreen}, the reduced operator $G'$ is indeed the inverse of some random Schrödinger operator $H'_\beta$ of the form~\eqref{eq:Form_H} with $W'$ defined as in~\eqref{eq:ReducedEdgeWeights}.

We therefore reorganize the path sum for $G(k,l)$ according to the skeleton of $\sigma$:
\begin{equation*}
    G(k,l)=\sum\limits_{m\geq 0}\sum\limits_{\substack{(k_q,i_q,j_q,k'_q)\\ 1\leq q\leq m}}\sum\limits_{\substack{\sigma:k\to l\\\text{sk}(\sigma)=(k_q,i_q,j_q,k'_q)_{1\leq q\leq m}}}\frac{W_{\sigma}}{(2\beta)_\sigma},
\end{equation*}
and show that it can be rewritten as
\begin{equation*}
    G(k,l)=\sum\limits_{m\geq 0}\sum\limits_{\substack{(k_q,k'_q)\\ 1\leq q\leq m}}\sum\limits_{\substack{\sigma':k\to l\\\text{sk}(\sigma)=(k_q,k'_q)_{1\leq q\leq m}}}\frac{W'_{\sigma'}}{(2\beta)_\sigma'}
\end{equation*}
according to the reduced skeleton of $\sigma'$, where we pose $W'$ as in~\eqref{eq:ReducedEdgeWeights} and $\beta'$ as in~\eqref{eq:Definition_Beta'u}, especially
\begin{equation*}
    \frac{1}{2\beta'_u}=\frac{1}{|U|^2}\sum_{i,j\in U}\widehat{G}^{U}(i,j).
\end{equation*}
Indeed, for a fixed reduced skeleton $(k_q,k'_q)_{1\leq q\leq m}$, summing over the admissible intermediate states $(i_q,j_q)_{1\leq q\leq m}$ of the original skeleton $(k_q,i_q,j_q,k'_q)_{1\leq q\leq m}$ yields (with $\sigma_{|\bullet}$ denoting the restriction of $\sigma$ to the corresponding interval)
\begin{equation*}
\begin{split}
    &\sum\limits_{\substack{\sigma:k\to l\\(i_q,j_q)_{1\leq q\leq m}}}\frac{W_{\sigma}}{(2\beta)_\sigma}\mathbf{1}_{\{\text{resk}(\sigma)=(k_q,k'_q)\}}\\
    ={}&\left(\sum\limits_{\sigma:k\to k_1}\frac{W_{\sigma_{|[k_0,k_1]}}}{(2\beta)_{\sigma_{|[k_0,k_1]}}}\right)\left(\sum_{i_1,j_1\in U}W_{k_1,i_1}\widehat{G}^{U}(i_1,j_1)W_{j_1,k'_1}\right)\left[\sum\limits_{\substack{\sigma:k'_1\to l\\(i_q,j_q)_{2\leq q\leq m}}}\frac{W_{\sigma}}{(2\beta)_\sigma}\mathbf{1}_{\{\text{resk}(\sigma)=(k_q,k'_q)_{2\leq q\leq m}\}}\right]\\
    ={}&\left(\sum\limits_{\sigma':k\to k_1}\frac{W_{\sigma'_{|[k_0,k_1]}}}{(2\beta')_{\sigma'_{|[k_0,k_1]}}}\right)\left(\frac{W'_{k_1,u}W'_{u,k'_1}}{2\beta'_u}\right)\left[\sum\limits_{\substack{\sigma:k'_1\to l\\(i_q,j_q)_{2\leq q\leq m}}}\frac{W_{\sigma}}{(2\beta)_\sigma}\mathbf{1}_{\{\text{resk}(\sigma)=(k_q,k'_q)_{2\leq q\leq m}\}}\right]\\
    ={}&\left(\sum\limits_{\sigma':k\to k'_1}\frac{W_{\sigma'_{|[k_0,k_1]}}}{(2\beta')_{\sigma'_{|[k_0,k_1]}}}\right)\left[\sum\limits_{\substack{\sigma:k'_1\to l\\(i_q,j_q)_{2\leq q\leq m}}}\frac{W_{\sigma}}{(2\beta)_\sigma}\mathbf{1}_{\{\text{resk}(\sigma)=(k_q,k'_q)_{2\leq q\leq m}\}}\right]\\
    ={}&\sum\limits_{\substack{\sigma':k\to l\\\text{resk}(\sigma)=(k_q,k'_q)}}\frac{W'_{\sigma'}}{(2\beta')_{\sigma'}},
\end{split}
\end{equation*}
where in the last equality, we iterate the same procedure for $2\leq q\leq m$.

Similar reduced path expansion by reorganization can be done for $G'(k,u)$ for $k\notin U$ and $G'(u,u)$ and the results are all in the above form. Therefore, by Remark~\ref{rema:RandomWalkGreen}, the inverse of $G'$ must be of the form~\eqref{eq:Form_H} where $\beta'_k=\beta_k$ if $k\notin U$ and $\beta'_u$ defined by formula~\eqref{eq:Definition_Beta'u} above.

\medskip

It remains to study the joint law of $(\beta'_k)_{k\in V'}$. To this end, we first calculate the joint Laplace transform of $G'(u,u)$ and $(\beta_k)_{k\notin U}$. Let $1_{U}\in\mathbb{R}^{V}$ be such that $1_U(i)=\mathbf{1}_{\{i\in U\}}$ and $\kappa>0$, $\lambda'\in\mathbb{R}^{V\setminus U}_+$ which we extend to $\lambda\in\mathbb{R}^{V}$ by adding null entries. Using~\eqref{eq:LaplaceTransform_MultiInverseGaussian} with $\theta=1$ and $\eta=\frac{\kappa 1_{U}}{|U|}$, the joint Laplace transform
\begin{equation}\label{eq:Beta_CoarseGrain_Laplace}
\begin{split}
    &\mathbb{E}\left[e^{-\frac{\kappa^2}{2}G'(u,u)-\langle\lambda',\beta\rangle_{V\setminus U}}\right]\\
    ={}&\mathbb{E}\left[e^{-\frac{1}{2}\langle\frac{\kappa 1_{U}}{|U|},G\frac{\kappa 1_{U}}{|U|}\rangle_{V}-\langle\lambda,\beta\rangle_{V}}\right]\\
    ={}&e^{-\kappa}\cdot e^{-\sum\limits_{k,l\notin U}W_{kl}(\sqrt{(1+\lambda_k)(1+\lambda_l)}-1)}\prod_{k\notin U}\frac{1}{\sqrt{1+\lambda_k}}
\end{split}
\end{equation}
factorizes, so $G'(u,u)$ is independent of $(\beta_{i})_{i\in V\setminus U}$ and follows the inverse Gamma$(\frac{1}{2})$ distribution.\footnote{The inverse Gamma$(\frac{1}{2})$ distribution has $\frac{\kappa^2}{2}$-Laplace transform $\int_{0}^{\infty}e^{-\frac{\kappa^2}{2}\frac{1}{2t}}\frac{1}{\sqrt{t}}e^{-t}\frac{1}{\Gamma(\frac{1}{2})}dt=e^{-\kappa}$.}

Now that we have seperated out the independent Gamma random variable $G'(u,u)$, we can use the $u$-field representation of the $H^{2|2}$-model to check that $(\beta'_k)_{k\in V'}$ is indeed the $H^{2|2}$-model on the reduced graph $\mathcal{G}'$. For this, first notice the following random walk expansion
\begin{equation*}
    \forall k\notin U,\quad \frac{G'(u,k)}{G'(u,u)}=\sum\limits_{\sigma':k\twoheadrightarrow u}\frac{W_{\sigma'}}{(2\beta')_{\sigma'}}=\sum\limits_{\sigma:k\twoheadrightarrow u}\frac{W_{\sigma}}{(2\beta)_{\sigma}}
\end{equation*}
where we sum over finite paths $\sigma'$ or $\sigma$ killed before hitting $u$. This follows from the same reorganization scheme as above: the random walk path sum calculates the contribution before the first hitting time at $u$, and $G'(u,u)$ corresponds to the large loop between the first visit and the last visit at $u$. Since $G'(u,u)$ is independent of $(\beta_{i})_{i\in V\setminus U}$ by the Laplace transform calculation above, it is independent of $(\frac{G'(u,k)}{G'(u,u)})_{k\notin U}$. But the display above also shows that the law of $(\frac{G'(u,k)}{G'(u,u)})_{k\notin U}$ is the same as the one generated by the $H^{2|2}$-model on the reduced graph, and the previous paragraph shows that the independent factor $G'(u,u)$ is distributed as an inverse Gamma$(\frac{1}{2})$ random variable. Therefore we recover exactly the $H^{2|2}$-model on the reduced graph $\mathcal{G}'$.
\end{proof}

\begin{rema}
As the first part of this proof is deterministic, one can derive this using linear algebra, and~\eqref{eq:Beta_CoarseGrain_Laplace} uses the specific distribution~\eqref{eq:Distribution_Beta} of the $H^{2|2}$-model to show that the coarse-grained model remains a $H^{2|2}$-model on the reduced graph. The algebraic procedure gives yet another proof of the result in~\cite{MR4517733}. However, the random walk expansion expression is also useful in establishing fine probabilistic estimates (see e.g.~\cite{Wang:2025aa}), and it is healthy to keep both proofs in mind in the sequel.
\end{rema}

\begin{proof}[Alternative proof of the first part of Theorem~\ref{th:CoarseGraining}]\label{proof:Alternative_CoarseGraining}
Let us indicate how to perform the algebraic computation in the above remark. Recall the representation~\eqref{eq:Form_H} of the random Schrödinger operator $H$ in the canonical basis $(e_k)_{k\in V}$. Perform an orthogonal change of basis in the subspace $E_u$ spanned by $(e_i)_{i\in U}$ by prescribing a new set of basis of $E_u$ with one component being $e_u=\frac{1}{|U|}\sum_{i\in U}e_i$. Consider the subspace $E'$ spanned by the reduced basis $(e_k)_{k\in V\setminus U}$ and $e_u$, and notice that the matrice $G'$ defined with~\eqref{eq:G'_G} is the restriction of $G$ to this subspace; more precisely, we define $G'x=P_{E'}(Gx)$ for $x\in E'$ where $P_{E'}$ is the orthogonal projection to $E'$. In particular, one checks that
\begin{equation*}
    G'(u,u)=\langle e_u,Ge_u\rangle=\frac{1}{|U|^2}\sum_{i,j\in U}\langle e_i,Ge_j\rangle=\frac{1}{|U|^2}G(i,j)
\end{equation*}
as in~\eqref{eq:G'_G}.

Now $H$ acts on $e_k$ with $k\in V\setminus U$ as
\begin{equation*}
    H(e_k)=-\sum_{i\in U}W_{ki}e_i+\sum_{l\neq u}H(k,l)e_l=-W'_{ku}e_u+\sum_{l\neq u}H(k,l)e_l\in E'
\end{equation*}
by the indistinguishability from outside assumption on $U$, so that $G'(-W'_{ku}e_u+\sum_{l\neq u}H(k,l)e_l)=P_{E'}(e_k)=e_k$. This implies that $H'(e_k)=-W'_{ku}e_u+\sum_{l\neq u}H(k,l)e_l$. Since $H'$ is a symmetric matrix, this determines all coefficients in $H'$ except for one, namely $H(u,u)=2\beta'_u$. To determine $H(u,u)$, use the Sherman-Morrison-Woodbury formula (or the Schur complement formula) for block matrix inversion which yields, with $1_{U}(i)=\mathbf{1}_{\{i\in U\}}$,
\begin{equation*}
    H(u,u)=\frac{|U|^2}{\langle 1_{U},(H_{U,U})^{-1}1_{U}\rangle}.
\end{equation*}

Therefore $H'=(G')^{-1}$ is of the form~\eqref{eq:Form_H}, especicially with~$\beta'_u$ as defined in~\eqref{eq:Definition_Beta'u}. This provides an alternative to the first part of the proof of Theorem~\ref{th:CoarseGraining}.
\end{proof}

\subsection{Fine-graining procedure}\label{subse:fine_graining_procedure}
We now state the key observation in this paper, that one can explicitly describe the inverse of the coarse-graining procedure above, which we call the fine-graining property of the $H^{2|2}$-model. We start with the reduced graph $\mathcal{G}'=(V',E',W')$ obtained by replacing a subset $U\subset V$ indistinguishable from outside by a single vertex $u$ with reduced weights~\eqref{eq:ReducedEdgeWeights}.
\begin{theo}\label{th:FineGrainBeta}
Consider a finite weighted graph $\mathcal{G}$ and let $\mathcal{G}'$ be its reduced graph obtained via a subset $U\subset V$ that is indistinguishable from outside. Consider the $H^{2|2}$-model (with boundary condition $\eta'$) on the reduced graph $\mathcal{G}'$ given by its random $\beta'$-field $(\beta'_k)_{k\in V'}$. Conditionally on $(\beta'_{k})_{k\in V'}$, one can sample the random vector $(\beta_k)_{k\in V}$ by introducing additional randomness and reconstruct the $H^{2|2}$-model (with boundary condition $\eta$ s.t. $\eta_{i}=\eta'_{i}$ for $i\notin U$ and $\eta_{i}= \frac{1}{|U|} \eta'_{u}$ for $i\in U$) on the original graph $\mathcal{G}$ explicitly.

Especially, the random vector $(\beta_i)_{i\in U}$ has an explicit conditional density as a function of $\{\beta'_k\}_{k\in V}$.
\end{theo}

\begin{proof}
By \eqref{eq:nu-eta-is-marginal} it is enough to consider the case without boundary condition.

Notice that by Theorem~\ref{th:CoarseGraining}, for $k\notin U$, it suffices to take $\beta_k=\beta'_k$, so that it only remains to study the conditional law of $(\beta_i)_{i\in U}$ given $(\beta'_k)_{k\in V'}$. We start by giving the explicit details in the case where $|U|=2$, i.e. splitting a vertex $u\in V'$ into two vertices indistinguishable from outside, then discuss briefly the general case as it follows in a similar manner.

Assume $|U|=2$. Without loss of generality, suppose that $U$ is indexed by $\{1,2\}$ so that it corresponds to the first two rows and columns in the matrix representation, and that the pinning $i_0$ is indexed last, i.e. represented by the last row and column. Suppose that the vertices $\{1,2\}$ is coarse-grained into one single vertex denoted by $1'$: we now calculate the law of $(\beta_1,\beta_2)$ given $\beta_{1'}$ and $(\beta_{k})_{k\notin\{1,2\}}$.

We can write out the marginal law of $(\beta_1,\beta_2)$ given $(\beta_{k})_{k\notin\{1,2\}}$ since their joint density is explicitly given by~\eqref{eq:Distribution_Beta}. The joint density is given exactly by the multivariate inverse Gaussian distribution of~\eqref{eq:MultiInverseGaussian}, with $W_{12}$ unchanged, $\theta_1=\theta_2=1$ and $\eta_1=\sum_{k\notin U}W_{1k}$, $\eta_2=\sum_{k\notin U}W_{2k}$. This can be easily verified using the Laplace transforms~\eqref{eq:LaplaceTransform_MultiInverseGaussian}. Indeed, if $\lambda_{1,2}>0$, $\lambda_{k}=0$ for $k\notin U$ and $(\beta_1,\beta_2,\dots,\beta_n)$ follows the law of the $H^{2|2}$-model on $V$ then
\begin{equation*}
    \mathbb{E}\left[e^{\langle\lambda,\beta\rangle}\right]=e^{-\sum_{i\sim j}W_{ij}(\sqrt{\lambda_i+1}\sqrt{\lambda_j+1}-1)}\frac{1}{\sqrt{\lambda_1+1}\sqrt{\lambda_2+1}},
\end{equation*}
while for $(\beta_1,\beta_2)$ following the above multivariate inverse Gaussian distribution with parameters $W_{12}$, $\theta_1=\theta_2=1$ and $\eta_1$, $\eta_2$,
\begin{equation*}
    \mathbb{E}\left[e^{\lambda_1\beta_1+\lambda_2\beta_2}\right]=e^{-W_{12}(\sqrt{\lambda_1+1}\sqrt{\lambda_2+1}-1)-\eta_1(\sqrt{\lambda_1+1}-1)-\eta_2(\sqrt{\lambda_2+1}-1)}\frac{1}{\sqrt{\lambda_1+1}\sqrt{\lambda_2+1}}.
\end{equation*}
One checks by inspection that these Laplace transforms are equal with our choice of $\eta_1,\eta_2$. It remains to further condition with respect to the extra information of $\beta_{1'}$ using the conditional density formula: by~\eqref{eq:Definition_Beta'u}, this only depends on $H_{|\{1,2\}}$.

More explicitly, with the boundary condition $\eta=(\eta_1,\eta_2)$ defined in the above paragraph, the condition joint law of $(\beta_1,\beta_2)$ knowing $(\beta'_{k})_{k\neq u}$ is proportional to
\begin{equation*}
    \mathbf{1}_{\{H_{U,U}>0\}}e^{-\frac{1}{2}\left(\langle 1,H_{U,U}1\rangle+\langle\eta,\widehat{G}^{U}\eta\rangle\right)}\frac{d\beta_{1}d\beta_2}{\sqrt{\det H_{U,U}}}
\end{equation*}
according to~\eqref{eq:MultiInverseGaussian}. Therefore the condition joint law of $(\beta_1,\beta_2)$ knowning furthermore $\beta_{1'}$ is obtained by restricting the above distribution to the level set defined by~\eqref{eq:Definition_Beta'u}. An explicit calculation in the case $|U|=2$ will be carried out later in the proof of Theorem~\ref{th:ExponentialMartingale}.

In the general case where $|U|$ is larger, the reasoning follows identical steps. Given $(\beta'_k)_{k\in V'}$, we already know $(\beta_{i})_{i\notin U}$, so that we know that marginal law of $(\beta_{i})_{i\in U}$ by conditioning to $(\beta_{i})_{i\notin U}=(\beta'_{k})_{k\neq u}$, which remains in the family of multivariate inverse Gaussian distributions of~\eqref{eq:MultiInverseGaussian} and whose explicit parameters can be deduced from the exact expression similar to the above Laplace transform calculation. The only extra information is $\beta'_{u}=F((\beta_i)_{i\in U})$, which is an explicit function of $(\beta_{i})_{i\in U}$ due to~\eqref{eq:Definition_Beta'u}, and it suffices to restrict the marginal law above to this slice.
\end{proof}

One can translate the above theorem into a fine-graining property using the random Schrödinger operator $H$ representation (equivalently, the random Green matrix $G=H^{-1}$) of the $H^{2|2}$-model. For $H$ it is immediate since its random data are exactly the diagonal entries, i.e. the random $\beta$-field. We now explain how to perform the fine-graining for the random operator $G=H^{-1}$.
\begin{theo}\label{th:FineGraining_GH}
Let $\mathcal{G}'$ be the reduced graph from $\mathcal{G}$ using a subset $U\subset V$ which is indistinguishable from outside. Consider the $H^{2|2}$-model on the reduced graph $\mathcal{G}'$ given by its random Green matrix $G'$. Conditionally on $G'$, we can sample $G$ on the original graph $\mathcal{G}$ by injecting independent random variables and reconstruct the $H^{2|2}$-model on the original graph $\mathcal{G}=(V,E,W)$.

Furthermore, one can perform the fine-graining procedure in such a way that $G'(k,l)=G(k,l)$ for all $k,l\notin U$.
\end{theo}

\begin{proof}
For simplicity, we give the detailed proof in the case $|U|=2$ that we use in the sequel, and the general case follows similarly.

Without loss of generality, suppose that $U=\{1,2\}$, so that we should perform the following reconstruction:
\begin{equation*}
\begin{pNiceArray}{c|ccc}
 G(1',1') & G(1',3) & \cdots & G(1',n) \\ \hline
  & G(3,3) & \cdots & G(3,n) \\
  &  & \ddots & \vdots \\
  &  &  & G(n,n) \\
\end{pNiceArray}
~\longrightarrow~
\begin{pNiceArray}{cc|ccc}
 G(1,1) & G(1,2) & G(1,3) & \cdots & G(1,n) \\
 G(2,1) & G(2,2) & G(2,3) & \cdots & G(2,n) \\ \hline
  &  & G(3,3) & \cdots & G(3,n) \\
  &  &  & \ddots & \vdots \\
  &  &  &  & G(n,n) \\
\end{pNiceArray}.
\end{equation*}
The coeffients with $i,j\notin U$ are the same by~\eqref{eq:G'_G}. The reconstruction can be done by inverting deterministically (i.e. pathwise, see Remark~\ref{rema:AlmostSure+Simultaneous}) $G'$ to get $H'$ and thus $(\beta'_k)_{k\in V'}$, then use Theorem~\ref{th:FineGrainBeta} to sample $(\beta_k)_{k\in V}$ and thus $H$, then invert $H$ to get $G$.

The fact that $G'(k,l)=G(k,l)$ for $k,l\notin U$ follows from the same argument for~\eqref{eq:G'_G} via the path expansion.
\end{proof}

The above description is a bit abstract and we now indicate the practical implementation of the above algorithm. We first record an auxiliary observation on the proportionality of rows/columns of the fine-grained matrix $G$, using the path expansion interpretation explained in Section~\ref{subse:green_function_and_random_walk_expansion_of_the_h_2_2}.
\begin{prop}\label{prop:Proportionality_G}
In the setting of Theorem~\ref{th:FineGraining_GH}, the random Green matrix $G$ on the fine-grained graph $\mathcal{G}=(V,E,W)$ has the following proportionality property: for all $k,l\notin U$ and all $i,j\in U$,
\begin{equation*}
    \frac{G(i,k)}{G(j,k)}=\frac{G(i,l)}{G(j,l)}=\frac{\sum_{v\in U}G(v,i)}{\sum_{v\in U}G(v,j)}.
\end{equation*}
\end{prop}
\begin{proof}
Recall the path expansion argument in the proof of Theorem~\ref{th:CoarseGraining} and consider $G(k,i)$ with $k\notin U$ and $i\in U$. Consider a path $\sigma$ from $k$ to $i$ in the large graph $\mathcal{G}$, and let $\circ$ be the first vertex the path $\sigma$ enters the indistinguishable subset $U$. Since $U$ is indistinguishable from outside, the distribution of $\circ$ is uniform over $U$, and the path expansion yields
\begin{equation*}
    G(k,i)=\sum_{\circ\in U}C(k,\circ)G(\circ,i).
\end{equation*}
where $C(k,\circ)$ is the contribution from paths starting from $k$ and stopped at the first time hitting the vertex $\circ$, i.e.
\begin{equation*}
    C(k,\circ)=2\beta_{\circ}\times \sum\limits_{\substack{\sigma:k\to \circ\\\text{killed at}~\circ}}\frac{W_{\sigma}}{(2\beta)_{\sigma}}.
\end{equation*}
By the indistinguashability, $C(k,\circ)$ is independent of $\circ\in U$ (notice that it only depends on the $\beta$-field outside $U$ and only one $W$-factor uses vertices in $U$). Furthermore, the same analysis applied to the smaller graph $\mathcal{G}'$ yields that this factor $C(k,\circ)$ is equal to $\frac{1}{|U|}\frac{G'(k,u)}{G'(u,u)}$, with $G'$ the given Green matrix on the reduced graph $\mathcal{G}'$ where $U$ is reduced to the single vertex $u$. Combining these observations we get that
\begin{equation*}
    \frac{G(k,i)}{\sum_{\circ\in U}G(\circ,i)}
\end{equation*}
is independent of the choice of $i\in U$, hence the desired result.
\end{proof}

\begin{rema}
When implementing the algorithm of Theorem~\ref{th:FineGraining_GH} to fine-grain the matrix $G'$ into $G$, one does not have to invert the full matrices of size $|V|$, but only a small submatrix of size $|U|$. We illustrate this with a practical example in the case $|U|=2$ above, say with $U=\{1,2\}$ and $u=1'$. By Remark~\ref{rema:AlmostSure+Simultaneous}, all coefficients $H(i,j)$ and $G(i,j)$ are known if $i,j\notin U$. To complete the matrix $G$, one only needs to sample the submatrix
\begin{equation*}
    \begin{pmatrix}G(1,1)&G(1,2)\\G(2,1)&G(2,2)\end{pmatrix}.
\end{equation*}
Indeed, knowing this submatrix, it suffices to fill in the rest of the first two columns to reconstruct the full matrix $G$. But this procedure is deterministic, as Proposition~\ref{prop:Proportionality_G} shows that the knowledge of the $2\times 2$ matrix above determines all quotients $\frac{G(1,k)}{G(2,k)}$ for all $k\notin\{1,2\}$. Together with the condition that $G(1,k)+G(2,k)=2G'(1',k)$ by~\eqref{eq:G'_G}, this determines all coefficients $G(1,k)$ and $G(2,k)$ with $k\notin\{1,2\}$.

It remains to sample the small $2\times 2$ matrix in the above display: one samples first $\widehat{G}^{U}=(H_{U,U})^{-1}$ as in~\eqref{eq:Definition_Beta'u} using the explicit law of $(\widehat{G}(1,1),\widehat{G}(1,2),\widehat{G}(2,2))$ conditionally on $\frac{1}{2\beta'_u}=\frac{1}{4}(\widehat{G}(1,1)+2\widehat{G}(1,2)+\widehat{G}(2,2))$ as in the proof of Theorem~\ref{th:FineGrainBeta}. More precisely, Theorem~\ref{th:FineGrainBeta} shows how to sample $(\beta_1,\beta_2)$ from the data of $(\beta'_k)_{k\in V'}$, so this determines the small matrix $H_{U,U}$ which we invert to sample $\widehat{G}^{U}$. Then a similar path expansion argument as in the previous proof of Proposition~\ref{prop:Proportionality_G} allows one to reconstruct $G_{U,U}$ from the matrix $\widehat{G}^{U}$. More precisely, consider a path $\sigma$ from $i\in\{1,2\}$ to $j\in\{1,2\}$. There are two possibilities: either $\sigma$ stays inside $U=\{1,2\}$, in which case the path expansion contribution is $\widehat{G}^{U}(i,j)$; or that $\sigma$ goes outside of $U$, in which case we denote by $\circ$ the last vertex in $U$ it visits before going out of $U$, and $\bullet$ the last time the path $\sigma$ returns to $U$. Similar considerations as in the above proof yield
\begin{equation*}
    G(i,j)=\widehat{G}^{U}(i,j)+\sum_{\circ,\bullet\in U}\widehat{G}^{U}(i,\circ)\frac{\beta_\circ}{\beta'_u}\frac{1}{|U|}G'(u,u)\frac{1}{|U|}\frac{\beta_{\bullet}}{\beta'_u}\widehat{G}^{U}(\bullet,j).
\end{equation*}
The gives an explicit expression of the random Green matrix $G$ in terms of the given reduced Green matrix $G'$ and the small random matrix $\widehat{G}^{U}$ that we sampled above. For example, with the subset $U=\{1,2\}$ reduced to the vertex $u=1'$,
\begin{equation*}
    G(1,2)=\widehat{G}^{U}(1,2)+\frac{1}{4}G'(1',1')\sum_{\circ,\bullet\in\{1,2\}}\frac{\beta_\circ \beta_\bullet}{(\beta'_{1'})^2}\widehat{G}^{U}(1,\circ)\widehat{G}^{U}(\bullet,2),
\end{equation*}
where especially $\frac{\beta_\circ \beta_\bullet}{(\beta'_{1'})^2}$ is also a function of $H_{U,U}$ only, thus involves only the inversion of the small matrix $\widehat{G}^{U}=(H_{U,U})^{-1}$.

Notice that the above procedure avoids the full inversion of the large $|V|\times |V|$ matrix, which reduces significantly the complexity in implementing the fine-graining algorithm for the random Green matrix established in Theorem~\ref{th:FineGraining_GH}.
\end{rema}

\subsection{Wired boundary coupling and infinite graph}\label{subse:wired_boundary_condition_and_infinite_graph}
In later applications, we need to consider the $H^{2|2}$-model on infinite graphs, so we now explain how to define this model as the thermodynamical limit of $H^{2|2}$-models on finite subgraphs with wired boundary conditions.

From now on, $\mathcal{G}=(V,E,W)$ is no longer assumed finite. If $U$ is a finite subset of $V$, one can define the graph $\mathcal{G}^{\widetilde{U}}=(\widetilde{U},\widetilde{E},\widetilde{W})$ on $\widetilde{U}=U\cup\{\delta\}$ with an additional vertex $\delta$ by setting
\begin{equation*}
\begin{cases}
    \widetilde{W}_{ij}=W_{ij} & i,j\in U \\ \widetilde{W}_{i,\delta}=\sum_{j:j\notin U}W_{ij} & i\in U.
\end{cases}
\end{equation*}

The graph $\mathcal{G}^{\widetilde{U}}$ being finite, one can consider the $H^{2|2}$-model on $\mathcal{G}^{\widetilde{U}}$ as in the previous sections. The law of the components $(\beta_{i})_{i\in U}$ of the \(H^{2|2}\)-model on $\mathcal{G}^{\widetilde{U}}$ is the \(H^{2|2}\)-model on the graph induced by $U$ with boundary condition $\eta_{i}= \widetilde{W}_{i,\delta}$, called the wired boundary condition. 

Take $(U_n)_{n\geq 0}$ an increasing exhausting sequence of subsets of $V$, then~\cite[Section~4]{MR3904155} shows that the $\beta$-field (of the \(H^{2|2}\)-model) on the infinite graph $\mathcal{G}$ can be defined as the Kolmogorov extension of the $\beta$-fields on $U_n$ with wired boundary condition. In particular, the distribution of the $\beta$-field on the infinite graph $\mathcal{G}$ is completely characterized by its finite dimensional marginals.

Such definition of the $\beta$ field is unique and can be done whenever all the \(W_{i} = \sum_{j\in V}W_{ij} \) are finite, even when the infinite graph is not locally finite.

However, the equivalence between $\beta$ field and $u$-field as described in \eqref{eq:RelationBeta_u} is no longer valid in the infinite graph setting. On an infinite graph, thermodynamic limit of the $u$-field is not always unique. There is a canonical choice of such limit as described in ~\cite{MR3904155}, using the wired boundary coupling, we recall such coupling in the case of hierarchical lattice in Section \ref{sec-wired-bd-hier}.

The situation is, however, more subtle when considering the fine-mesh scaling limit. We will examine this in detail in Section~\ref{sec:continuous_space_scaling_limit_of_the_hiarachical_h_2_2} for the case of the $H^{2|2}$-model on the Dyson hierarchical lattice. Indeed, we are required to use a coupling that is fundamentally different from the wired boundary one. To prepare for this, we introduce a necessary martingale property in the next subsection.

\subsection{A one-step martingale property}
This subsection follows the spirit of~\cite[Section~5]{MR3904155}, albeit with a different graph enlargement and consequently a different filtration: to establish the martingale property, it is convenient to start by proving a stronger exponential martingale property using Laplace transform formulas.
\begin{theo}\label{th:ExponentialMartingale}
Consider the $H^{2|2}$-model on a finite graph $\mathcal{G}=(V,E,W)$, and let $U\subset V$ be indistinguishable from outside. Define the random $\beta'$-field on the reduced graph $V'=(V\setminus U)\cup\{u\}$ with the coarse-graining procedure defined in Section~\ref{subse:coarse_graining_procedure}, and let $\mathcal{F}'$ be the $\sigma$-algebra generated by $(\beta'_k)_{k\in V'}$. Let $\eta\in\mathbb{R}^{V}_{+}$ be a vector such that $\eta_i=\eta_j$ for all $i,j\in U$. Then
\begin{equation*}
    \forall \lambda\in\mathbb{R}^{V}_+,\quad \mathbb{E}\left[e^{-\langle\lambda,G\eta\rangle-\frac{1}{2}\langle\lambda,G\lambda\rangle}~|~\mathcal{F}'\right]=e^{-\langle\widehat{\lambda},G'\widehat{\eta}\rangle-\frac{1}{2}\langle\widehat{\lambda},G'\widehat{\lambda}\rangle}
\end{equation*}
where we set $\widehat{\lambda}_{i}=\lambda_i, \widehat{\eta}_i=\eta_i$ if $i\ne u$, and $\widehat{\lambda}_u=\frac{1}{\sqrt{|U|}}\sum\limits_{k\in U}\lambda_{k}$, $\widehat{\eta}_u=\frac{1}{\sqrt{|U|}}\sum\limits_{k\in U}\eta_{k}$.
\end{theo}

\begin{proof}
We give the detailed proof in the case $|U|=2$, which is the case that we need in the sequel, and the general case is similar. Without loss of generality, supppose that $U=\{1,2\}$ is represented by the first two components of $V=\{1,2,\dots,n-1,\infty\}$ and the pinning point is $\infty$ and is represented by the last component in the matrix representation of the random Schrödinger operator $H$, so that
\begin{equation*}
   H = \begin{pmatrix}
    2\beta_1 & -W_{12} & \cdots & \cdots & -W_{1\infty} \\
    -W_{12} & 2 \beta_2 & \ddots &  \ddots & -W_{2\infty} \\
    \vdots & \ddots & 2\beta_3 & \ddots & \vdots \\
    \vdots &  \ddots & \ddots & \ddots & \vdots \\
    -W_{1\infty} & -W_{2\infty} & \cdots & \cdots & 2\beta_\infty \\
    \end{pmatrix}
\end{equation*}
and $G=H^{-1}$ the associated random Green matrix. By~\eqref{eq:G'_G}, the coarse-grained random Green matrix $G'$ is obtained from averaging, so we define the symmetric matrix
\begin{equation*}
   P = \begin{pNiceArray}{cc|ccc}
    -\frac{1}{\sqrt{2}} & \frac{1}{\sqrt{2}} & 0 & \cdots & 0 \\
    \frac{1}{\sqrt{2}} & \frac{1}{\sqrt{2}} & 0 & \cdots & 0 \\ \hline
    0 & 0 &  &  &  \\
    \vdots & \vdots &  & I_{n-2} &  \\
    0 & 0 &  &  &  \\
    \end{pNiceArray}=P^{-1}=P^{t}
\end{equation*}
in such a way that\footnote{Compared to the alternative algebraic proof to Theorem~\ref{th:CoarseGraining} above, the change of basis is off by a multiplicative constant $\sqrt{2}$, in such a way that we don't need extra notations and normalizing constants for the inverse of $P$.}
\begin{equation*}
    \widetilde{G}=PGP=\begin{pmatrix}B & A^{t} \\ A & G'\end{pmatrix}
\end{equation*}
where $G'$ is the coarse-grained Green matrix $G'$ and $B\in\mathbb{R}$ whose value is unimportant.

Consider now $H'=(G')^{-1}$ the coarse-grained random Schrödinger operator on the reduced graph, and by Section~\ref{subse:coarse_graining_procedure} it is represented by the matrix of the form~\eqref{eq:Form_H} with diagonal entries $(2\beta'_{1'},2\beta_3,\dots,2\beta_{\infty})$, where
\begin{equation*}
    \frac{1}{2\beta'_{1'}}=\frac{1}{4}\frac{2\beta_1+2\beta_2+2W_{12}}{4\beta_1\beta_2-(W_{12})^2}
\end{equation*}
by~\eqref{eq:Definition_Beta'u} and $W'$ is the reduced weights~\eqref{eq:ReducedEdgeWeights}. This matrix is naturally associated to the the matrix $\widetilde{H}=(\widetilde{G})^{-1}=PHP$ via the Schur complement formula:
\begin{equation*}
    \widetilde{G}=\begin{pmatrix}1 & A^{t}H' \\ 0 & I_{n-1}\end{pmatrix}\begin{pmatrix}\frac{1}{2\check{\beta}_{1'}} & 0 \\ 0 & G'\end{pmatrix}\begin{pmatrix}1 & 0 \\ H'A & I_{n-1}\end{pmatrix}
\end{equation*}
and
\begin{equation*}
    \widetilde{H}=\begin{pmatrix}1 & 0 \\ -H'A & I_{n-1}\end{pmatrix}\begin{pmatrix}2\check{\beta}_{1'} & 0 \\ 0 & H'\end{pmatrix}\begin{pmatrix}1 & -A^{t}H' \\ 0 & I_{n-1}\end{pmatrix}
\end{equation*}
where we defined $2\check{\beta}_{1'}=\beta_1+\beta_2+W_{12}$ since the upper left corner of $\widetilde{H}$ is
\begin{equation*}
    \begin{pmatrix}-\frac{1}{\sqrt{2}} & \frac{1}{\sqrt{2}} \\ \frac{1}{\sqrt{2}} & \frac{1}{\sqrt{2}}\end{pmatrix}\begin{pmatrix}2\beta_1 & -W_{12} \\ -W_{12} & 2\beta_2\end{pmatrix}\begin{pmatrix}-\frac{1}{\sqrt{2}} & \frac{1}{\sqrt{2}} \\ \frac{1}{\sqrt{2}} & \frac{1}{\sqrt{2}}\end{pmatrix}=\begin{pmatrix}2\check{\beta}_{1'} & -(\beta_1-\beta_2) \\ -(\beta_1-\beta_2) & \beta_1+\beta_2-W_{12}\end{pmatrix}.
\end{equation*}

Notice that the vector $A^{t}H'$ appearing in the Schur complement formula above is almost zero except for one coordinate thanks to the assumption that $U$ is indistinguishable from outside. Indeed, using the formula for $\widetilde{H}$ above we have
\begin{equation*}
    \widetilde{H}=\begin{pmatrix}2\check{\beta}_{1'} & -2\check{\beta}_{1'}A^{t}H' \\ -2\check{\beta}_{1'}H'A & H'\end{pmatrix}
\end{equation*}
but on the other hand, the first row of $\widetilde{H}$ must be
\begin{equation*}
    (2\check{\beta}_{1'},-(\beta_1-\beta_2),0,\dots,0)
\end{equation*}
by the calculation of the upper left corner of $\widetilde{H}$ above. The zeros appear in the above display since $\widetilde{H}=PHP$ so that if $k\notin\{1,2\}$, the $k$-th coordinate in the first row of $\widetilde{H}$ is $0$ as $W_{1k}=W_{2k}$.

We now perform the following change of $\beta$-field variables:\footnote{One checks the elementary equality $(\beta_1-\beta_2)^{2}=(\beta_1+\beta_2+W)^2-(4\beta_1\beta_2-W^2)-2W(\beta_1+\beta_2+W)$.}
\begin{equation*}
\begin{split}
    d\beta_1d\beta_2d\beta_3\dots d\beta_\infty&=\frac{\beta_1+\beta_2+W_{12}}{2|\beta_1-\beta_2|}d\check{\beta}_{1'}d\beta'_{1'}d\beta_3\dots d\beta_\infty\\
    &=\frac{\check{\beta}_{1'}}{\sqrt{4\check{\beta}_{1'}^2-2\beta'_{1'}\check{\beta}_{1'}-4W_{12}\check{\beta}_{1'}}}d\check{\beta}_{1'}d\beta'_{1'}d\beta_3\dots d\beta_\infty,
\end{split}
\end{equation*}
and rewrite the probability density~\eqref{eq:MultiInverseGaussian} with general parameters $(\theta,\kappa)$
\begin{equation*}
    d\nu^{W,\theta,\kappa}=\left(\sqrt{\frac{2}{\pi}}\right)^{V}\mathbf{1}_{\{H>0\}}e^{-\frac{1}{2}\left(\langle\theta,H\theta\rangle+\langle\kappa,G\kappa\rangle-2\langle\theta,\kappa\rangle\right)}\frac{\prod_{i\in V}\theta_i}{\sqrt{\det H}}\prod_{i\in V}d\beta_{i}
\end{equation*}
in the new $\beta$-variables and the reduced matrices $H',G'$. For convenience we denote by
\begin{equation*}
    \frac{1}{\Xi}=\frac{\beta_1+\beta_2+W_{12}}{2|\beta_1-\beta_2|}=\frac{\check{\beta}_{1'}}{\sqrt{4\check{\beta}_{1'}^2-2\beta'_{1'}\check{\beta}_{1'}-4W_{12}\check{\beta}_{1'}}}
\end{equation*}
the (reciprocal of the) Jacobian above, and notice that the last expression is a function of solely $\beta'_{1'}$ and $\check{\beta}_{1'}$. This quantity also appears in the Schur complement formulas above since
\begin{equation*}
    A^{t}H'=\left(\frac{\beta_1-\beta_2}{2\check{\beta}_{1'}},0,\dots,0\right)=(\epsilon^{12}\frac{\Xi}{2},0,\dots,0).
\end{equation*}
where we denote by $\epsilon^{12}=\text{sgn}(\beta_1-\beta_2)$.

Using $G=P\widetilde{G}P$, the Schur complement formula and the expression for $A^{t}H'$ above,
\begin{equation}\label{eq:QuadForm_G}
    \langle \kappa,G\kappa\rangle=\frac{(\kappa_1-\kappa_2)^{2}}{4\check{\beta}_{1'}}+\langle(\widehat{\kappa}+\epsilon^{12}\kappa^{\Xi}),G'(\widehat{\kappa}+\epsilon^{12}\kappa^{\Xi})\rangle
\end{equation}
where $\widehat{\kappa}=(\frac{1}{\sqrt{2}}(\kappa_1+\kappa_2),\kappa_3,\dots,\kappa_\infty)$ and $\kappa^{\Xi}=(\frac{\Xi}{2\sqrt{2}}(\kappa_1-\kappa_2),0,\dots,0)$. Similarly
\begin{equation*}
    \langle \theta,H\theta\rangle=\frac{\check{\beta}_{1'}}{4}((\theta_1-\theta_2)+\epsilon^{12}\Xi(\theta_1+\theta_2))^2+\langle\widehat{\theta},H'\widehat{\theta}\rangle,
\end{equation*}
and if furthermore we assume $\theta_1=\theta_2$ then
\begin{equation*}
    \langle \theta,H\theta\rangle=\frac{\check{\beta}_{1'}}{4}\Xi^2(\theta_1+\theta_2)^2+\langle\widehat{\theta},H'\widehat{\theta}\rangle.
\end{equation*}
Always under the assumption that $\theta_1=\theta_2$, the probability density $d\nu^{W,\theta,\kappa}$ above is written as
\begin{equation}\label{eq:ReducedMultivariateDensity}
    e^{-\frac{1}{2}\left(\frac{\check{\beta}_{1'}}{4}\Xi^2(\theta_1+\theta_2)^2+\langle\widehat{\theta},H'\widehat{\theta}\rangle+\frac{(\kappa_1-\kappa_2)^{2}}{4\check{\beta}_{1'}}+\langle(\widehat{\kappa}+\epsilon^{12}\kappa^{\Xi}),G'(\widehat{\kappa}+\epsilon^{12}\kappa^{\Xi})\rangle-2\langle\theta,\kappa\rangle\right)}\frac{\mathbf{1}_{\{H>0\}}\prod\limits_{i\in V}\sqrt{\frac{2}{\pi}}\theta_i}{\sqrt{\det H'}\sqrt{2\check{\beta}_{1'}}\Xi}d\check{\beta}_{1'}d\beta'_{1'}d\beta_3\dots d\beta_\infty
\end{equation}
where $\epsilon^{12}$ is an independent standard Rademacher random variable corresponding to a random permutation that expresses the indistinguishability of the vertices $\{1,2\}=U$. Notice that this probability density still belongs to the family~\eqref{eq:MultiInverseGaussian} if $\epsilon^{12}$ is sampled.

We finally come back to the Laplace transform in the statement of the theorem, where the law of $H$ is the above with $\theta=1$ and $\kappa=0$. Conditionally on $\mathcal{F}'$, the variables $(\beta'_{1'},\beta_3,\dots,\beta_{\infty})$ are all sampled, so that by polarizing~\eqref{eq:QuadForm_G} we have, with the notation $\widehat{\alpha}=(\frac{1}{\sqrt{2}}(\alpha_1+\alpha_2),\alpha_3,\dots,\alpha_\infty)$ and $\eta^{\Xi}=0$ since we assumed $\eta_1=\eta_2$,
\begin{equation*}
\begin{split}
    &\mathbb{E}\left[e^{-\langle\lambda,G\eta\rangle-\frac{1}{2}\langle\lambda,G\lambda\rangle}~|~\mathcal{F}'\right]\\
    ={}&\mathbb{E}\left[e^{-\langle\widehat{\lambda}+\epsilon^{12}\lambda^{\Xi},G'\widehat{\eta}\rangle-\frac{1}{2}\langle\widehat{\lambda}+\epsilon^{12}\lambda^{\Xi},G'(\widehat{\lambda}+\epsilon^{12}\lambda^{\Xi})\rangle-\frac{(\lambda_1-\lambda_2)^2}{8\check{\beta}_{1'}}}~\Big|~\mathcal{F}'\right]\\
    ={}&e^{-\langle\widehat{\lambda},G'\widehat{\eta}\rangle-\frac{1}{2}\langle\widehat{\lambda},G'\widehat{\lambda}\rangle}\mathbb{E}\left[e^{-\langle\widehat{\lambda}+\widehat{\eta},G'\epsilon^{12}\lambda^{\Xi}\rangle-\frac{1}{2}\langle\lambda^{\Xi},G'\lambda^{\Xi}\rangle-\frac{(\lambda_1-\lambda_2)^2}{8\check{\beta}_{1'}}}~\Big|~\mathcal{F}'\right].
\end{split}
\end{equation*}
It remains to show that the last conditional expectation is $1$ for any value of $\epsilon^{12}$: this is because the reweighting factor $e^{-\langle\widehat{\lambda}+\widehat{\eta},G'\epsilon^{12}\lambda^{\Xi}\rangle-\frac{1}{2}\langle\lambda^{\Xi},G'\lambda^{\Xi}\rangle-\frac{(\lambda_1-\lambda_2)^2}{8\check{\beta}_{1'}}}$ changes the law of $\check{\beta}_{1'}$, more precisely from the conditional law of $\check{\beta}_{1'}$ in~\eqref{eq:ReducedMultivariateDensity} with $\theta=1,\kappa=0$ to the conditional law of $\check{\beta}_{1'}$ in~\eqref{eq:ReducedMultivariateDensity} with $\theta=1,\kappa=\lambda+\eta$. Indeed, with $\theta=1$, the conditional density of $\check{\beta}_{1'}$ given $H'$ using~\eqref{eq:ReducedMultivariateDensity} with $\kappa=0$ is proportional to
\begin{equation*}
    e^{-\frac{1}{2}\check{\beta}_{1'}\Xi^2}\frac{\mathbf{1}_{\{\check{\beta}_{1'}>0\}}}{\sqrt{2\check{\beta}_{1'}}\Xi}d\check{\beta}_{1'},
\end{equation*}
while the conditional density of $\check{\beta}_{1'}$ given $H'$ using~\eqref{eq:ReducedMultivariateDensity} with $\kappa=\lambda+\eta$ (in particular $\kappa^{\Xi}=\lambda^{\Xi}$ and $\kappa_1-\kappa_2=\lambda_1-\lambda_2$) is proportional to
\begin{equation*}
    e^{-\frac{1}{2}\check{\beta}_{1'}\Xi^2-\frac{(\kappa_1-\kappa_2)^2}{8\check{\beta}_{1'}}-\frac{1}{2}\langle\lambda^{\Xi},G'\lambda^{\Xi}\rangle-\langle\widehat{\kappa},G'\epsilon^{12}\lambda^{\Xi}\rangle}\frac{\mathbf{1}_{\{\check{\beta}_{1'}>0\}}}{\sqrt{2\check{\beta}_{1'}}\Xi}d\check{\beta}_{1'}.
\end{equation*}
The difference between the two conditional densities is exacty the exponential reweighting factor above, and the two proportionality constant must be equal since both $d\nu^{W,1,\eta}$ and $d\nu^{W,1,\kappa}$ are probability measures.
\end{proof}

\begin{rema}
Using \eqref{eq:nu-eta-is-marginal}, with minor modifications, the above theorem can be generalized to the family of measures $\mu^{W,\theta,\eta}$ defined with~\eqref{eq:MultiInverseGaussian} under the condition that $\theta_i=\theta_j$ for all $i,j\in U$.
\end{rema}

As a consequence, we obtain the following:
\begin{coro}\label{coro:Exp(u)Martingale}
Consider the $H^{2|2}$-model on a finite graph $\mathcal{G}=(V,E,W)$, and let $U\subset V$ be indistinguishable from outside, and $i_0\in V\setminus U$ that serves to define the pinning condition for the $u$-field, i.e.
\begin{equation*}
    \forall k\in V,\quad e^{u_k}=\frac{G(k,i_0)}{G(i_0,i_0)}.
\end{equation*}
Recall that $\frac{1}{G(i_0,i_0)}$ is distributed as $2\gamma$ where $\gamma$ is a standard Gamma distributed random variable. Consider also the corresponding quantities (denoted by the superscript $'$) on the reduced graph $V'=(V\setminus U)\cup\{u\}$ with the coarse-graining procedure, and let $\mathcal{F}'$ be the $\sigma$-algebra generated by $(\beta'_k)_{k\in V'\setminus\{i_0\}}$ and $\gamma$. Then
\begin{equation*}
    \forall k\in V\setminus U,\quad \mathbb{E}\left[e^{u_k}|\mathcal{F}'\right]=e^{u'_k}\quad\text{and}\quad\forall k\in U,\quad \mathbb{E}\left[e^{u_k}|\mathcal{F}'\right]=e^{u'_{u}}.
\end{equation*}
where the last term $e^{u'_u}$ is the exponential of the $u'$-field at the reduced vertex $u$.
\end{coro}

\begin{proof}
We first sample $G(i_0,i_0)=\frac{1}{2\gamma}$ which is independent of $(\beta'_k)_{k\in V'\setminus\{i_0\}}$, and consider $\eta=\frac{1}{G(i_0,i_0)}1_{U}$. It suffices to derive the main identity in Theorem~\ref{th:ExponentialMartingale} with respect to $\lambda$ in the direction $e_k$ then take $\lambda=0$.
\end{proof}

\begin{rema}
The preceding proofs yield explicit expressions for the $\beta$-fields in the fine-graining calculation of Theorem~\ref{th:FineGrainBeta}. In particular, the explicit conditional law of $\check{\beta}_{1'}$ given the reduced random Schrödinger operator $H'$ (described in the proof of Theorem~\ref{th:ExponentialMartingale}), combined with an independent standard Rademacher random variable $\epsilon^{12}$, allows for the construction of the operator $H$ on the original graph. This constitutes a constructive implementation of the fine-graining process of Theorem~\ref{th:FineGrainBeta}. The fine-graining of the random Green matrix $G'$ in Theorem~\ref{th:FineGraining_GH} is obtained analogously.
\end{rema}

\section{The $H^{2|2}$-model on the Dyson hierarchical lattice}\label{sec:the_h_2_2}
We are now ready to apply the exact real-space renormalization results derived previously to the $H^{2|2}$-model on the Dyson hierarchical lattice~\cite{MR436850}.

\subsection{Definition of the hierarchical $H^{2|2}$-model}
\label{sec-wired-bd-hier}
The Dyson hierarchical lattice $\mathbb{X}$, also referred to as the homogeneous hierarchical structure of degree $2$, is an ultrametric space indexed by the positive integers. Throughout this paper, we abusively write $\mathbb{X}=\{1,2,\dots\}$ to represent the underlying set; this notation is justified as the hierarchical structure imposed on $\mathbb{X}$ is fixed. Specifically, the distance $d_{\mathbb{X}}(i,j)$ between two points is defined as the smallest non-negative integer $n$ such that $i$ and $j$ belong to the same interval of the form $Q^{(n)}_k=[k2^{n}+1,(k+1)2^{n}]$ for some integer $k \ge 0$, see Figure \ref{fig1}. For example,
\begin{equation*}
    d_{\mathbb{X}}(1,1)=0, d_{\mathbb{X}}(1,3)=2, d_{\mathbb{X}}(1,9)=4, \dots
\end{equation*}

We define the edge weights on $\mathbb{X}$ by fixing a parameter $\rho>1$ and letting
\begin{equation}\label{eq:Defi_HierarchicalEdgeWeights}
    \forall i,j\in \mathbb{X},\quad W_{ij}=\overline{W}(2\rho)^{-d_{\mathbb{X}}(i,j)},
\end{equation}
where $\overline{W}$ is a global constant identified as the inverse temperature. This defines the weighted graph Laplacian $\Delta_{W}$ of $\mathbb{X}$ as in Section~\ref{subse:definition_of_the_h_2_2}, which has spectral dimension $d_s=2\frac{\ln 2}{\ln\rho}$ (see e.g.~\cite[Proposition~1.3]{MR2276652}). It is useful to think of $\mathbb{X}$ as an analogue of a $d_s$-dimensional Euclidean space. While we do not need the precise definition here, $d_s$ controls the behavior of the integrated density of states of $\Delta_{W}$ near the ground state, behaving similarly to the Laplacian on $\mathbb{Z}^d$ with spatial dimension $d$. Moreover, the simple random walk on $\mathbb{X}$ exhibits a phase transition at $d_s=2$ (recurrence vs.\ transience), mirroring the critical dimension of random walks on $\mathbb{Z}^d$.

We now define the $H^{2|2}$-model on the infinite graph $\mathbb{X}$ using the finite approximation scheme recalled in Section~\ref{subse:wired_boundary_condition_and_infinite_graph}. By~\cite[Proposition~1]{MR3904155}, this is equivalent to defining an infinite length random vector $(\beta_i)_{i\in\mathbb{X}}$ such that for any finite subset $V$ of $\mathbb{X}$, the following Laplace transform holds: if $\lambda\in\mathbb{R}^{\mathbb{X}}$ is such that $\lambda_i=0$ for all $i\notin V$, then
\begin{equation}
\label{laplace-transform-betaX}
    \mathbb{E}\left[e^{-\langle\lambda,\beta\rangle}\right]=e^{-\sum_{i,j\in\mathbb{X}}W_{ij}(\sqrt{(\lambda_i+1)(\lambda_j+1)}-1)}\prod_{i\in V}\frac{1}{\sqrt{\lambda_i+1}}.
\end{equation}
Indeed, the exponent of the exponential in the right-hand side above is finite and is equal to
\begin{equation*}
    \sum_{i,j\in V}W_{ij}(\sqrt{(\lambda_i+1)(\lambda_j+1)}-1)+\sum_{i\in V}(\sum_{j\in\mathbb{X}\setminus V}W_{ij})(\sqrt{\lambda_i+1}-1),
\end{equation*}
where $\sum_{j\in\mathbb{X}\setminus V}W_{ij}$ is the pinning weight for the wired boundary condition when we perform the effective reduction by replacing all vertices outside $V$ by a single additional vertex $\delta$.

In what follows, we construct the $H^{2|2}$-model on $\mathbb{X}$ by considering an exhaustion of the graph by finite balls $\Lambda_{n}=\{1,\dots,2^{n}\}$ of radius $n\geq 0$ containing $1$. To implement the wired boundary condition, we add an auxiliary vertex $\delta_n$ and denote by $\widetilde{\Lambda}_n=\Lambda_{n}\cup\{\delta_n\}$ the vertex set of the augmented graph. The edge weights are taken to be $W_{ij}$ if $i,j\in \Lambda_n$, while the connections to the boundary node are given by
\begin{equation*}
    W_{i,\delta_n}=\sum_{j\in\mathbb{X}\setminus\Lambda_n}W_{ij}=\overline{W}\frac{\rho^{-n}}{2(\rho-1)}.
\end{equation*}
Based on the results in Section~\ref{subse:wired_boundary_condition_and_infinite_graph}, we consider the $H^{2|2}$-model defined on the augmented graph $\widetilde{\Lambda}_n$. The marginal distribution of this model on the subset $\Lambda_n$ constitutes the $H^{2|2}$-model on $\Lambda_n$ with wired boundary conditions; this corresponds to the law of the components $(\beta_i)_{i\in\Lambda_n}$ of the random vector defined on $\widetilde{\Lambda}_n$. Finally, the $H^{2|2}$-model on the infinite graph $\mathbb{X}$ is defined via the Kolmogorov extension theorem applied to the consistent sequence of these $H^{2|2}$-models on $\Lambda_n$ equipped with wired boundary conditions (where the vertex $\delta_n$ effectively captures the influence of the complement $\mathbb{X}\setminus\Lambda_n$).

Given the random vector $(\beta_i)_{i\in\widetilde{\Lambda}_n}$ on $(\widetilde{\Lambda}_{n},W)$, we define the random Schrödinger operator $H^{(n)}$ acting on $\ell^{2}(\widetilde{\Lambda}_n)$ following Section~\ref{subse:definition_of_the_h_2_2} by
\begin{equation*}
    \forall i\in\widetilde{\Lambda}_n,\quad H^{(n)}f(i)=2\beta_if(i)-\sum_{j\in\widetilde{\Lambda}_n}W_{ij}f(j).
\end{equation*}
The random Green matrix $G^{(n)}=(H^{(n)})^{-1}$ has a path sum expansion as explained in Section~\ref{subse:green_function_and_random_walk_expansion_of_the_h_2_2}. Moreover, we may employ the $u$-field representation of the $H^{2|2}$-model reviewed in Section~\ref{subse:coupling_with_different_pinning_positions}, with pinning at the additional vertex $\delta_n$. We define
\begin{equation}
\label{eqpsi-n-i}
    \psi^{(n)}_{i} := \begin{cases} e^{u^{(n)}_i}=\frac{G^{(n)}(i,\delta_n)}{G^{(n)}(\delta_n,\delta_n)} & i\in \Lambda_{n} \\ 1 & \text{otherwise} .\end{cases}
\end{equation}
By \cite[Theorem 1.ii]{MR3904155}, for any $i$, $(\psi^{(n)}_{i})_{n\geq 0}$ is a ${\mathcal{F}}_{n}$ martingale where ${\mathcal{F}}_{n}=\sigma(\beta_{i},i\in \Lambda_{n})$. The a.s. martingale limits are denoted by \(\psi_{i} \).

The $\beta$-field is recovered from the $u$-field via~\eqref{eq:RelationBeta_u}. In particular, we have
\begin{equation*}
    2\beta_{\delta_n}=\sum_{j\in\widetilde{\Lambda}_n}W_{j\delta_n}e^{u_j^{(n)}}+2\gamma,
\end{equation*}
where $\gamma$ is a standard Gamma random variable representing the independent component $\gamma=\frac{1}{2G^{(n)}(\delta_n,\delta_n)}$. In summary, this construction relies solely on $(\beta_i)_{i\in\Lambda_n}$ and an independent Gamma random variable $\gamma$.

\begin{rema}[Wired boundary coupling]\label{rema:PinningConditionCoupling}
In particular, for any fixed spectral dimension $d_s$ and inverse temperature $\overline{W}$, one can choose the same $\gamma$ for all $(\widetilde{\Lambda}_n)_{n\geq 0}$ and then sample the $u^{(n)}$-fields independently to recover the $H^{2|2}$-model $H^{(n)}$ on $\widetilde{\Lambda}_n$. This yields a coupling of the sequence $(H^{(n)})_{n\geq 0}$ which corresponds to a special case of the wired boundary coupling applied to the Dyson hierarchical lattice.
\end{rema}

\subsection{Coarse-graining procedure on the hierarchical $H^{2|2}$-model}
Let $(H^{(n)})_{n}$ be the $H^{2|2}$-models on $\widetilde{\Lambda}_{n}$ under the wired boundary coupling described in the previous remark. In the sequel, we denote the additional vertex $\delta_n$ by $\infty$, so that $\widetilde{\Lambda}_n=\{1,\dots,2^{n},\infty\}$ for $n\geq 0$. Recall that this coupling involves a single independent inverse Gamma$(\frac{1}{2})$ variable associated with the pinning vertex $\infty$, which is common to all $\widetilde{\Lambda}_n$. Since the component at $\infty$ is fixed across levels, defining the $H^{2|2}$-models on $(\widetilde{\Lambda}_n,W)$ with edge weights~\eqref{eq:Defi_HierarchicalEdgeWeights} reduces to specifying the $u$-field, or equivalently $(e^{u^{(n)}_i})_{i\in \Lambda_n}$, on $\Lambda_{n}$. 

Consequently, the coarse-graining procedure described in~\cite[Corollary~2.3]{MR4517733}, defined by setting
\begin{equation*}
    \forall i\in\Lambda_{n-1},\quad e^{\widetilde{u}^{(n-1)}_i}=\frac{1}{2}(e^{u^{(n)}_{2i-1}}+e^{u^{(n)}_{2i}}),
\end{equation*}
provides a way to pass from $H^{(n)}$ to a random Schrödinger matrix $\widetilde{H}^{(n-1)}$ acting on $\ell^{2}(\widetilde{\Lambda}_{n-1})$. It follows from~\eqref{eq:ReducedEdgeWeights} that $\widetilde{H}^{(n-1)}$ is an $H^{2|2}$-model on $\widetilde{\Lambda}_{n-1}$ with inverse temperature $\frac{2\overline{W}}{\rho}$; moreover, the variables $(\widetilde{u}^{(n-1)}_i)_{i\in\Lambda_{n-1}}$ follow the law of the $u$-field on $\widetilde{\Lambda}_{n-1}$ with inverse temperature $\frac{2\overline{W}}{\rho}$. 

The coupling of $H^{(n)}$ and $\widetilde{H}^{(n-1)}$ induced by the coarse-graining procedure (which we refer to in the sequel as the {\em fine-graining coupling}) differs from the wired boundary coupling of Remark~\ref{rema:PinningConditionCoupling}. If $\rho\ne 2$, i.e. $d_{s}\ne 2$, the inverse temperature is modified, whereas for the wired boundary coupling, the inverse temperature remains unchanged.

Even in the critical dimension $\rho=2$, i.e. $d_{s}=2$, the two couplings differ, although the inverse temperature remains invariant in both cases. In fact, for fixed $n$, $\widetilde{H}^{(n-1)}$ and $H^{(n-1)}$ are equal in distribution.

\begin{rema}
An alternative proof of this fact follows directly from the exponential martingale property of Theorem~\ref{th:ExponentialMartingale}; see Corollary~\ref{coro:Exp(u)Martingale}.
\end{rema}

Our coarse-graining procedure in Section~\ref{subse:coarse_graining_procedure} is slightly more general, as it allows for the explicit coarse-graining of the entire matrix $H^{(n+1)}$ (resp.\ $G^{(n+1)}$) to $\widetilde{H}^{(n)}$ (resp.\ $\widetilde{G}^{(n)}$). More precisely, let $H^{(n+1)}$ be the random Schrödinger operator defining the $H^{2|2}$-model on $(\widetilde{\Lambda}_{n+1},W)$ with edge weights~\eqref{eq:Defi_HierarchicalEdgeWeights}, and let $G^{(n+1)}$ be the random Green matrix obtained by inverting $H^{(n+1)}$. We define the random matrix $\widetilde{G}^{(n)}$ by
\begin{equation*}
    \widetilde{G}^{(n)}(i,j)=\frac{1}{4}\left(G^{(n+1)}(2i-1,2j-1)+G^{(n+1)}(2i-1,2j)+G^{(n+1)}(2i,2j-1)+G^{(n+1)}(2i,2j)\right)
\end{equation*}
for all $i,j\in\widetilde{\Lambda}_{n}$. Here, for $i=\infty$, we adopt the convention $2i-1=2i=\infty$, and similarly for $j$. In other words, we average over adjacent pairs of columns (excluding the last column associated with the vertex $\infty$) and perform the corresponding operation on the rows. Finally, deterministically inverting $\widetilde{G}^{(n)}$ yields $\widetilde{H}^{(n)}$, thereby recovering the random Schrödinger operator representation.

The above coarse-graining procedure can also be made explicit in the $\beta$-field representation. Indeed, recall from Section~\ref{subse:wired_boundary_condition_and_infinite_graph} that, under the wired boundary condition, it suffices to specify $(\beta_k)_{k\in\Lambda_{n+1}}$ to define the $H^{2|2}$-model on $(\widetilde{\Lambda}_{n+1},W)$. Now given the random $\beta$-potential $(\beta_{k})_{k\in\Lambda_{n+1}}$ of $H^{(n+1)}$, we can define via~\eqref{eq:Definition_Beta'u}
\begin{equation*}
    \forall k\in\Lambda_{n},\quad \frac{1}{2\widetilde{\beta}'_k}=\frac{1}{4}\frac{2\beta_{2k-1}+2\beta_{2k}+2W_{2k-1,2k}}{4\beta_{2k-1}\beta_{2k}-(W_{2k-1,2k})^{2}},
\end{equation*}
which is exactly the random $\beta$-potential needed to define $\widetilde{H}^{(n)}$ on $\widetilde{\Lambda}_n$ with inverse temperature $\frac{2\overline{W}}{\rho}$ under the wired boundary.

\subsection{Fine-graining coupling on the hierarchical $H^{2|2}$-model}
\label{sec:find-grain-coupling}
Our fine-graining results in Section~\ref{subse:fine_graining_procedure} establishes the reverse procedure, i.e. one can explicitly construct $H^{(n+1)}$ (equivalently $G^{(n+1)}$) given $\widetilde{H}^{(n)}$ (equivalently $\widetilde{G}^{(n)}$) by adding appropriate extra randomness.

Fix $\rho>1$. Let us define the fine-graining coupling of the $H^{2|2}$-model on the hierarchical lattice, i.e., the sequence of random Schrödinger operators $(\widetilde{H}^{(n)})_{n}$, obtained by iterating the fine-graining procedure described in Theorem~\ref{th:FineGrainBeta} or alternatively Theorem~\ref{th:FineGraining_GH}. Note that, unlike the wired boundary coupling, in this sequence of models the inverse temperature is non-constant; at each step, a factor $\frac{2}{\rho}$ is applied to the inverse temperature.

More precisely, one first samples the $H^{2|2}$-model on the $(2^{0}+1)$-point graph $(\widetilde{\Lambda}_0,W^{(0)})$. Then, using either of the above-mentioned theorems, one obtains the $H^{2|2}$-model on the $(2^{1}+1)$-point graph $(\widetilde{\Lambda}_1,W^{(1)})$ by blowing up the vertex $\{1\}$ into $\{1,2\}$, where the new edge weight $W^{(1)}$ is computed using $W^{(0)}$. We repeat this procedure, computing $W^{(n+1)}$ using $W^{(n)}$ at each iteration; for instance, applying the theorems twice blows up $\{1,2\}$ into $\{\{1,2\},\{3,4\}\}$, yielding the $H^{2|2}$-model on the $(2^{2}+1)$-point graph $(\widetilde{\Lambda}_2,W^{(2)})$, and so on. This in particular provides a coupling of the $H^{2|2}$-models at all levels $(\widetilde{\Lambda}_n,W^{(n)})_{n\geq 0}$. This coupling will be called the fine-graining coupling on the Dyson hierarchical lattice $\mathbb{X}$. By definition, for all $n$, $\widetilde{G}^{(n)}(\infty,\infty)=\widetilde{G}^{(0)}(\infty,\infty) = \frac{1}{2\gamma}$, i.e. the Gamma random variable $\gamma$ at the pinning point is common to all $\widetilde{\Lambda}_{n}$.

\begin{rema}[Change of the inverse temperature under the fine-graining procedure]\label{rema:FineGrainingCoupling}
This coupling is not the same as the wired boundary coupling discussed in Section~\ref{subse:fine_graining_procedure} as we keep blowing up internal vertices instead of only adding new vertices from outside. Indeed, the fine-graining procedure multiplies the inverse temperature $\overline{W}$ of by a factor $\frac{\rho}{2}$ every time we pass from $\widetilde{\Lambda}_n$ to $\widetilde{\Lambda}_{n+1}$, as the fine-graining coupling inverts the coarse-graining procedure.
\end{rema}

The most interesting case, which we will discuss in detail in the sequel, is that of $d_s=2$ (i.e. $\rho=2$). Indeed, when $d_s\neq 2$, the fine-graining procedure yields a coupling of $H^{2|2}$-models with inverse temperatures that vary with the level $n\geq 0$. As $n$ tends to infinity (approximating $\mathbb{X}$), the inverse temperature converges to either $0$ or $\infty$, leading to a limiting object with degenerate parameters. However, in the case $d_s=2$, the fine-graining procedure preserves the inverse temperature. Consequently, it is meaningful to investigate the almost sure convergence properties under the fine-graining coupling of Remark~\ref{rema:FineGrainingCoupling}, since the potential limiting object possesses, a priori, a non-trivial inverse temperature. We will rigorously define the limiting object and establish its fundamental properties in Section~\ref{sec:continuous_space_scaling_limit_of_the_hiarachical_h_2_2} below.

Finally, we make some useful remarks concerning the $u$-fields under the fine-graining coupling. To this end, we first define a filtration $(\widetilde{\mathcal{F}}_n)_{n\geq 0}$ as follows: we set $\widetilde{\mathcal{F}}_0=\sigma(\gamma)$, and let $\widetilde{\mathcal{F}}_{n}$ be generated by $\gamma$ and the random potentials $(\widetilde{\beta}^{(n)}_i)_{i\in\Lambda_n}$ of $\widetilde{H}_{n}$. The exponential martingale property of Theorem~\ref{th:ExponentialMartingale} holds with respect to this filtration. In particular, Corollary~\ref{coro:Exp(u)Martingale} implies that, with $ e^{\widetilde{u}^{(n)}_i}=\frac{\widetilde{G}^{(n)}(i,\delta_n)}{\widetilde{G}^{(n)}(\delta_n,\delta_n)}$, we have
\begin{equation}\label{eq:Martingale_u_Hierarchical}
    \mathbb{E}\left[e^{\widetilde{u}^{(n+1)}_k}\,\middle|\,\widetilde{\mathcal{F}}_n\right]=e^{\widetilde{u}^{(n)}_i}\quad\text{if}\quad k\in\{2i-1,2i\}.
\end{equation}
We will see later that this equation provides a natural mechanism to define martingales indexed by $x\in [0,1]$.

Note that quantities under the fine-graining coupling are denoted with a tilde: $\widetilde{\beta}^{(n)}_{i}$, $\widetilde{u}^{(n)}_{i}$, $\widetilde{H}^{(n)}$ and $\widetilde{G}^{(n)}$, whereas quantities under the wired boundary coupling are denoted without a tilde: $\beta^{(n)}_{i}$, $u^{(n)}_{i}$, $H^{(n)}$ and $G^{(n)}$. In the case $\rho=2$, i.e. $d_{s}=2$, at fixed level $n$, they are equal in distribution.

\subsection{Some known estimates of the hierarchical $H^{2|2}$-model}
We now recall some known localization results on the hierarchical $H^{2|2}$-model with edge weights $W_{ij}=\overline{W}(2\rho)^{-d_{\mathbb{X}}(i,j)}$ under the wired boundary coupling of Section~\ref{subse:wired_boundary_condition_and_infinite_graph}. Recall that we denoted by $d_s=2\frac{\ln 2}{\ln\rho}$ the spectral dimension.

Recall that $e^{u^{(n)}_i}=\frac{G^{(n)}(i,\delta_n)}{G^{(n)}(\delta_n,\delta_n)}$. The following fractional moment estimates are proven in~\cite[Theorem~3]{Wang:2025aa}: $\forall 0<s<\frac{1}{2}$ and $\forall n\geq 1$,
\begin{enumerate}
    \item If $d_s<2$, there exists $C(\overline{W},d_s,s)>0$ such that for all $i\in\Lambda_n$,
    \begin{equation*}
        \mathbb{E}\left[e^{su^{(n)}_i}\right]\leq C(\overline{W},d_s,s)\rho^{-sn}.
    \end{equation*}
    \item If $d_s=2$, there exists $C(\overline{W},s)>0$ such that for all $i\in\Lambda_n$,
    \begin{equation}\label{eq:WangZeng}
        \mathbb{E}\left[e^{su^{(n)}_i}\right]\leq C(\overline{W},d_s,s)C(\overline{W},s)^{-sn}.
    \end{equation}
    where in particular $C(\overline{W},s)\to 2>1$ as $\overline{W}\to 0$.
\end{enumerate}

\begin{rema}
The estimates above can be extended to all fractional moments with $0<s<1$ by generalizing the Ward identity arguments of~\cite[Appendix~C]{MR2728731} or~\cite[Section~2.5]{MR4021254}, which yields
\begin{equation*}
    \mathbb{E}[e^{su^{(n)}_i}]=\mathbb{E}[e^{(1-s)u^{(n)}_i}],
\end{equation*}
which is valid for all $s\in\mathbb{R}$. The only point missing is $s=\frac{1}{2}$ but this can be recovered by a standard continuity argument. That is, by Hölder's inequality,
\begin{equation*}
    \mathbb{E}[e^{\frac{1}{2}u^{(n)}_i}]\leq\mathbb{E}[e^{\frac{1+\epsilon}{2}u^{(n)}_i}]^{\frac{1}{1+\epsilon}}\leq C\rho^{-\frac{1-\epsilon}{2(1+\epsilon)}n},
\end{equation*}
for any $\epsilon>0$.
\end{rema}

\section{Continuous space scaling limit of the hierarchical $H^{2|2}$-model}\label{sec:continuous_space_scaling_limit_of_the_hiarachical_h_2_2}

We now construct a continuum scaling limit of the hierarchical $H^{2|2}$-model embedded in the interval $[0,1]$ at the critical dimension $d_s=2$ (equivalently, $\rho=2$). Our construction draws inspiration from the theory of Gaussian Multiplicative Cascades developed by Mandelbrot~\cite{MR431351,MR431352} and Kahane-Peyrière~\cite{MR431355}; specifically, we utilize the martingale property established in Theorem~\ref{th:ExponentialMartingale} and Corollary~\ref{coro:Exp(u)Martingale} to define a positive measure-valued martingale and analyze its limit. We establish fundamental properties of the limiting measure and derive characterization theorems in terms of Vertex Reinforced Jump Processes on the hierarchical lattice. It must be emphasized, however, that this $H^{2|2}$-limiting random measure is fundamentally distinct from the class of Gaussian Multiplicative Chaos random measures.

\subsection{Construction of the limiting measure}
Fix inverse temperature $\overline{W}>0$, consider the fine-graining coupling of \(H^{2|2}\)-models on $\widetilde{\Lambda}_{n},\ n\geq 0$, defined in Section \ref{sec:find-grain-coupling}. In particular, for each $n$, we have a field $\widetilde{u}^{(n)}_{i},\ i\in \Lambda_{n}$. Using these fields, we define a random measure $\mathfrak{m}_n$ on $[0,1]$ that is absolutely continuous with respect to the Lebesgue measure $\lambda$, with the Radon-Nikodym derivative given by
\begin{equation}
\label{eq:def-mnx}
    \frac{d\mathfrak{m}_n}{d\lambda}(x)=e^{\widetilde{u}^{(n)}_i}\quad\text{if}\quad x\in\left[\frac{i-1}{2^{n}},\frac{i}{2^{n}}\right).
\end{equation}

For a fixed $x \in [0,1]$, let $i_n^x$ be the unique index such that $x$ lies in the dyadic interval $[\frac{i_{n}^{x}-1}{2^{n}}, \frac{i_{n}^x}{2^{n}})$. Define $\varphi^{(n)}_{x}:=e^{\widetilde{u}^{(n)}_{i_{n}^{x}}}$ By \eqref{eq:Martingale_u_Hierarchical}, the sequence of random variables $\varphi^{(n)}_{x}$ forms a martingale, which therefore converges almost surely to a random variable denoted by $\varphi_{x}$, i.e.
\begin{equation}
\forall x\in [0,1], \ \lim_{n\to \infty} \varphi^{(n)}_{x} = \varphi_{x} \text{ a.s.}
\label{eq:mart-eux}
\end{equation}
See Figure~\ref{fig2} for an illustration of this construction.

\begin{figure}[h]
    \centering
    \scalebox{0.8}{

\begin{tikzpicture}[
    scale=1.5,
    xscale=8, 
    yscale=0.6, 
    lattice/.style={draw=gray!20, line width=0.3pt}, 
    pathline/.style={draw=red!90!black, line width=1.5pt},
    interval/.style={|-|, line width=0.8pt}
]

\def\numlevels{6} 
\def\targetX{0.68} 

\foreach \n in {0,...,\numlevels} {
    \pgfmathsetmacro{\numintervals}{2^\n}
    \foreach \i in {1,...,\numintervals} {
        \pgfmathsetmacro{\xstart}{(\i-1)/\numintervals}
        \pgfmathsetmacro{\xend}{\i/\numintervals}
        \pgfmathsetmacro{\xmid}{(\xstart + \xend)/2}
        
        \draw[lattice, |-|] (\xstart, -\n) -- (\xend, -\n);
        
        \ifnum\n>0
            \pgfmathsetmacro{\parentI}{ceil(\i/2)}
            \pgfmathsetmacro{\parentN}{\n-1}
            \pgfmathsetmacro{\numintervalsParent}{2^\parentN}
            \pgfmathsetmacro{\xstartP}{(\parentI-1)/\numintervalsParent}
            \pgfmathsetmacro{\xendP}{\parentI/\numintervalsParent}
            \pgfmathsetmacro{\xmidP}{(\xstartP + \xendP)/2}
            \draw[lattice] (\xmidP, -\parentN) -- (\xmid, -\n);
        \fi
    }
}

\pgfmathsetmacro{\currentStart}{0}
\pgfmathsetmacro{\currentEnd}{1}

\foreach \n in {0,...,\numlevels} {
    \pgfmathsetmacro{\currentMid}{(\currentStart + \currentEnd)/2}
    
    \draw[pathline, |-|] (\currentStart, -\n) -- (\currentEnd, -\n);
    
    \ifnum\n<\numlevels
        \pgfmathsetmacro{\nextLevel}{\n+1}
        \pgfmathparse{\targetX < \currentMid ? 1 : 0}
        \ifnum\pgfmathresult=1
            \pgfmathsetmacro{\nextStart}{\currentStart}
            \pgfmathsetmacro{\nextEnd}{\currentMid}
        \else
            \pgfmathsetmacro{\nextStart}{\currentMid}
            \pgfmathsetmacro{\nextEnd}{\currentEnd}
        \fi
        \pgfmathsetmacro{\nextMid}{(\nextStart + \nextEnd)/2}
        
        \draw[pathline] (\currentMid, -\n) -- (\nextMid, -\nextLevel);
        
        \global\let\currentStart\nextStart
        \global\let\currentEnd\nextEnd
    \fi
}

\draw[->, thick] (-0.05, -\numlevels - 1.0) -- (1.05, -\numlevels - 1.0) node[right] {};
\draw (0, -\numlevels - 0.9) -- (0, -\numlevels - 1.1) node[below] {\scriptsize 0};
\draw (1, -\numlevels - 0.9) -- (1, -\numlevels - 1.1) node[below] {\scriptsize 1};

\fill[red!90!black] (\targetX, -\numlevels - 1.0)  node[below=1pt] {$x$};

\foreach \n in {0,...,\numlevels} {
    \node[left, font=\scriptsize, gray] at (-0.05, -\n) {$n=\n$};
}

\end{tikzpicture}

 }
\caption{Sequence of $\varphi^{(n)}_{x}=e^{\widetilde{u}^{(n)}_{i}}$ converges to $\varphi_{x}$}
\label{fig2}
\end{figure}
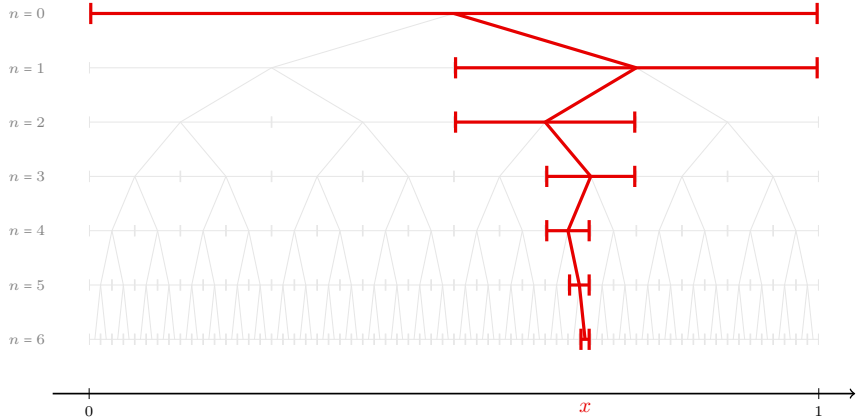

Corollary~\ref{coro:Exp(u)Martingale} immediate implies that the sequence of random measure $m_n$ is a positive measure-valued martingale on $[0,1]$, in the sense that for any Borel measurable set $A\subset[0,1]$, one has
\begin{equation*}
\mathbb{E}[\mathfrak{m}_{n+1}(A)\,|\,\widetilde{\mathcal{F}}_{n}]=\mathfrak{m}_n(A),
\end{equation*}
where we recall that $\widetilde{\mathcal{F}}_n$ is generated by $\gamma$ and $(\widetilde{u}^{(n)}_{i})_{1\leq i\leq 2^{n}}$, and $\widetilde{\mathcal{F}}_0=\sigma(\gamma)$. As a consequence of martingale convergence theorem, there exists an almost sure limit $\lim_{n\to \infty}\mathfrak{m}_{n}$ in the weak topology of measures.

\begin{defi}[Random measure limit of the hierarchical $H^{2|2}$-model at the critical dimension]
Fix the inverse temperature $\overline{W}>0$. We denote by $\mathfrak{m}_{\mathbb{X}}$ the almost-sure limit of the positive measure-valued martingale $(\mathfrak{m}_n)_{n\geq 0}$ defined in \eqref{eq:def-mnx} with respect to the filtration $\widetilde{\mathcal{F}}_n$, and call it the continuous space scaling limit of the two-dimensional hierarchical $H^{2|2}$-model at inverse temperature $\overline{W}$.
\end{defi}

Contrary to the Gaussian multiplicative chaos measure, the passage of $\mathfrak{m}_n$ to $\mathfrak{m}_{n+1}$ follows a different mechanism. Indeed, under the fine-graining coupling, which by Remark~\ref{rema:AlmostSure+Simultaneous} is realization-wise, the total mass of $\mathfrak{m}_n$ on any interval $[\frac{i-1}{2^{n}},\frac{i}{2^{n}})$ is almost surely the same as that of $\mathfrak{m}_{n+1}$, $\mathfrak{m}_{n+2}$, \dots In particular, the total mass $\mathfrak{m}_{\mathbb{X}}([0,1])$ is $\widetilde{\mathcal{F}}_0$-measurablbe.
\begin{prop}[Non-degeneracy]
The measure $\mathfrak{m}_{\mathbb{X}}$ is almost surely non-degenerate.
\end{prop}
\begin{proof}
By construction, $\mathfrak{m}_0([0,1])=e^{\widetilde{u}^{(0)}_1}$, and according to the above discussion, the total mass remains unchanged realization-wise as $n$ increases, therefore $\mathfrak{m}_{\mathbb{X}}([0,1])$ is also equal to $e^{\widetilde{u}^{(0)}_1}$, which is almost surely non-trivial.
\end{proof}

Actually, since the distribution of $e^{\widetilde{u}^{(0)}_1}$ is explicit, we already know about the exact law of the total mass $\mathfrak{m}_{\mathbb{X}}([0,1])$.
\begin{prop}[Exact law of the total mass]
The total mass $\mathfrak{m}_{\mathbb{X}}([0,1])$ of the limiting measure of the hierarchical $H^{2|2}$-model is distributed as an inverse Gaussian random variable $\text{IG}(1,\frac{\overline{W}}{2(\rho-1)})$. Especially, $\mathbb{E}[\mathfrak{m}_{\mathbb{X}}([0,1])]=1$.

Furthermore, for any non-trivial interval $[x,y]\subset[0,1]$, we have for all $p\in\mathbb{R}$
\begin{equation*}
    \mathbb{E}[\mathfrak{m}_{\mathbb{X}}([x,y])^{p}]<\infty.
\end{equation*}
\end{prop}
\begin{proof}
The distribution of $\mathfrak{m}_{\mathbb{X}}([0,1])=e^{\widetilde{u}^{(0)}_1}$ can be read off directly from the exact distribution~\eqref{eq:Distribution_u}, with $W_{1,\delta_0}=\overline{W}\frac{\rho^{-n}}{2(\rho-1)}$. From this exact distribution or alternatively by the Ward identity~\cite[Appendix~C]{MR2728731} or~\cite[Section~2.5]{MR4021254}, $\mathbb{E}[e^{\widetilde{u}^{(0)}_1}]=1$.

The positive moments of $\mathfrak{m}_{\mathbb{X}}([0,1])$ are finite since there are upper bounded by $\mathbb{E}[\mathfrak{m}_{\mathbb{X}}([0,1])^{p}]$, and an inverse Gaussian distribution has all positive moments. For the negative moments, pick a small enough dyadic interval $[(i-1)2^{-n},i2^{-n})\subset[x,y]$, with $\mathfrak{m}_n([(i-1)2^{-n},i2^{-n}))$ distributed as $2^{n}e^{\widetilde{u}^{(n)}_i}$. By Ward identity, for any $p<0$, $\mathbb{E}[e^{p\widetilde{u}^{(n)}_i}]=\mathbb{E}[e^{(1-p)\widetilde{u}^{(n)}_i}]\leq\mathbb{E}[e^{(1-p)\mathfrak{m}_{\mathbb{X}}([0,1])}]<\infty$ by the above discussion on the positive moments. Therefore all negative moments of $\mathfrak{m}_{\mathbb{X}}([x,y])$ exist as $\mathfrak{m}_{\mathbb{X}}([(i-1)2^{-n},i2^{-n}))=\mathfrak{m}_{n}([(i-1)2^{-n},i2^{-n}))$ almost surely.
\end{proof}

\subsection{Ergodicity and a zero-one law}

In this section, we establish ergodicity properties of the $H^{2|2}$-model on the hierarchical lattice $\mathbb{X}$ via a mixing property. We then derive a recurrence/transience dichotomy via a zero-one law, establishing a result analogous to~\cite[Proposition~3]{MR3904155} in the context of the hierarchical lattice. While we employ conceptual ideas similar to those in~\cite{MR3904155}, the distinct graph structure of $\mathbb{X}$ requires an alternative proof strategy. Specifically, the argument in~\cite{MR3904155} relies on the finite degree of the graph $\mathcal{G}$ to exploit the $1$-dependence property of the random $\beta$-field~\cite[Proposition~1]{MR3904155}. Since the hierarchical lattice is topologically an infinite complete graph, this dependence structure does not hold, motivating the specific approach presented here.

We work under the assumption that $d_s=2$, so that the inverse temperature $\overline{W}$ remains constant under either the wired boundary condition coupling or the fine-graining coupling.

\subsubsection{Mixing and ergodicity}
Consider the sequence of automorphisms $\{g_n\}_{n\geq 0}$ of the infinite hierarchical lattice $\mathbb{X}$, where $g_0$ is the identity map. For any $n\geq 1$, $g_n$ is defined by exchanging the cluster $[1,2^{n}]$ with $[2^{n}+1,2^{n+1}]$ (via the natural translation $i\mapsto i+2^{n}$ modulo $2^{n+1}$) while keeping the other vertices fixed. Explicitly, we define $g_n(i)=i+2^{n}$ if $i\in[1,2^{n}]$, $g_n(i)=i-2^{n}$ if $i\in [2^{n}+1,2^{n+1}]$, and $g_n(i)=i$ if $i>2^{n+1}$. Observe that $g_n=g_n^{-1}$ and that for any two vertices $i,j\in\mathbb{X}$, $W_{g_n(i),g_n(j)}=W_{ij}$.

The group of automorphisms of $\mathbb{X}$ is generated by the sequence $(g_{n})_{n \geq 0}$ (details can be found in \cite[Appendix A]{km2012}). It is a countable, discrete group of probability measure-preserving maps for the \(H^{2|2}\)-model. Consequently, the ergodicity and mixing conditions discussed in e.g. \cite[Chap. 2]{KerrLi16} are applicable in this setting. We are now ready to establish the ergodicity of the \(H^{2|2}\)-model on the hierarchical lattice.

\begin{prop}
\label{prop:ergodic}
Consider the \(H^{2|2}\)-model on $\mathbb{X}$, under the distribution $\nu^{W}$, the random variables $(\beta_{i})_{i\in \mathbb{X}}$ (which's law is characterized by \eqref{laplace-transform-betaX}), and martingale limit $(\psi_{i})_{i\in \mathbb{X}}$ (from \eqref{eqpsi-n-i}) are stationary and ergodic for the group of automorphisms of $\mathbb{X}$ generated by $(g_{n})_{n\geq 0}$.
\end{prop}

\begin{proof}
Let \(\mathbf{G}\) be the group of automorphisms of \(\mathbb{X}\) generated by the countable discrete sequence \((g_{n})_{n \geq 1}\). By definition of the \(H^{2|2}\)-model, every element \(h\in \mathbf{G}\) is measure-preserving.

For the \(\beta\)-field, rather than establishing the strong mixing property for the entire group \(\mathbf{G}\), we demonstrate that the generating sequence \((g_n)_{n\geq 1}\) satisfies a mixing property sufficient for our purposes. Specifically, we show that for all events \(E,F\) measurable with respect to \(\sigma((\beta_i)_{i\in\mathbb{X}})\),
\begin{equation}\label{eq:StrongMixing}
    \lim\limits_{n\to\infty}\left(\mathbb{P}[g_n^{-1}(F)\cap E]-\mathbb{P}[g_n^{-1}(F)]\mathbb{P}[E]\right)=0.
\end{equation}
Here, the translated event \(g_n^{-1}(F)\) is defined via the action on random variables: for any \(\sigma((\beta_i)_{i\in\mathbb{X}})\)-measurable function \(f\), we define \((g_n^{-1}f)((\beta_{i})_{i\in\mathbb{X}})=f((\beta_{g_n^{-1}(i)})_{i\in\mathbb{X}})\), and in particular this applies to \(f=\mathbf{1}_{F}\).

We can approximate the events $E,F$ by respective events $E',F'$ depending only on finitely many vertices, i.e. for any $\epsilon>0$, we can find such $E',F'$ satisfying $\mathbb{P}[E\Delta E']<\epsilon$ and $\mathbb{P}[F\Delta F']<\epsilon$. Suppose that $E'$ depends on the finite vertex set $A$ and $F'$ depends on the finite vertex set $B$. Since $g_n$ is measure-preserving, this also implies that $\mathbb{P}[g_n^{-1}(F)\,\Delta\,g_n^{-1}(F')]<\epsilon$. Therefore if we can establish the strong mixing property for $E',F'$, i.e.
\begin{equation}\label{eq:Finite_StrongMixing}
    \lim\limits_{n\to\infty}\left(\mathbb{P}[g_n^{-1}(F')\cap E']-\mathbb{P}[g_n^{-1}(F')]\mathbb{P}[E']\right)=0,
\end{equation}
then by taking $\epsilon$ to $0$, all terms converge to their infinite counterparts and we get~\eqref{eq:StrongMixing}.

Now suppose that $A,B\subset [1,2^{l}]$ for some large $l$, so that $g_n^{-1}(F)$ depends on vertices in $g_n^{-1}(B)$. For any $n\geq l$, the vertices in $g_n^{-1}(B)$ and $A$ are all connected by the same edge weights $\epsilon_n=\overline{W}(2\rho)^{-n}$ uniformly, and this sequence goes to $0$ as $n$ goes to infinity. Denote by $p_{\epsilon_n}(\beta)$ the joint density of the random $\beta$-field in the set $A\cup B$ by~\eqref{eq:Distribution_Beta}, so that
\begin{equation*}
    \mathbb{P}[g_n^{-1}(F')\cap E']=\int\mathbf{1}_{\{g_n^{-1}(F')\}}\mathbf{1}_{\{E'\}}p_{\epsilon_n}(\beta)d\beta.
\end{equation*}
As $\epsilon_n$ goes to $0$, $p_{\epsilon_n}(\beta)$ goes pointwise to $p_0(\beta)$ which is the density corresponding to the same configuration with the vertices in $g_n^{-1}(B)$ and $A$ disconnected. Since we are dealing with positive probability densities, by Scheffé's lemma, $||p_{\epsilon_n}(\beta)-p_0(\beta)||_{L^1}\to 0$ as $\epsilon_n$ goes to $0$, so that
\begin{equation*}
    \left|\int\mathbf{1}_{\{g_n^{-1}(F')\}}\mathbf{1}_{\{E'\}}(p_{\epsilon_n}(\beta)-p_0(\beta))d\beta\right|\leq ||p_{\epsilon_n}(\beta)-p_0(\beta)||_{L^1}\to 0
\end{equation*}
as $\epsilon_n$ goes to $0$. Finally, by the $1$-dependence property~\cite[Proposition~1]{MR3904155}, the $\beta$-field on $g_n^{-1}(B)$ and on $A$ are independent under the distribution $p_0(\beta)$, and factorization yields
\begin{equation*}
    \int\mathbf{1}_{\{g_n^{-1}(F')\}}\mathbf{1}_{\{E'\}}p_0(\beta)d\beta=\mathbb{P}[g_n^{-1}(F')]\mathbb{P}[E'].
\end{equation*}
This concludes the proof of~\eqref{eq:Finite_StrongMixing}, therefore~\eqref{eq:StrongMixing} by the previous discussion.

Now let $E\in \mathcal{F}$ be a $\mathbf{G}$-invariant event, for every $n\geq 1$, $\mathbb{P}(g_{n}^{-1}(E)\cap E)= \mathbb{P}(E)$. By Strong mixing \eqref{eq:StrongMixing}, we also have
$$\lim_{n\to \infty} \mathbb{P} (g_{n}^{-1}(E)\cap E) = \mathbb{P}(E)^{2} .$$
Therefore, $\mathbb{P}(E)(1-\mathbb{P}(E))=0$, which implies that $\mathbb{P}(E)\in \{0,1\}$, we have proved ergodicity of $\beta$ field.

For the martingale limits $(\psi_{i})_{i\in \mathbb{X}}$, by the explicit almost sure limit definition of $\psi_{i}$, if $A\in \mathcal{B}(\mathbb{R}^{\mathbb{X}})$ is $\mathbf{G}$-invariant, then the set $\{\beta: (\psi_{i})_{i\in \mathbb{X}} \in A\}$ is also $\mathbf{G}$-invariant, hence, $A$ has measure 0 or 1 under the law of $(\psi_{i})_{i\in \mathbb{X}}$, we have proved ergodicity of $\psi$ field.
\end{proof}

\begin{coro}
On the Dyson hierarchical lattice $\mathbb{X}$ with spectral dimension $d_s=2$, the $\widetilde{\beta}$-field and the $\widetilde{u}$-field (defined in Sec. \ref{sec:find-grain-coupling}) of the $H^{2|2}$-model are stationary and ergodic with respect to the group of automorphisms of $\mathbb{X}$. Especially, we have the following zero-one law on the limit martingale (defined in \eqref{eq:mart-eux}:
\begin{equation*}
    \mathbb{P}\left[\left\{\forall x\in[0,1), \varphi_{x}=0\right\}\right],  \mathbb{P}\left[\left\{\forall x\in[0,1), \varphi_{x}>0\right\}\right]\in\{0,1\}.
\end{equation*}
\end{coro}
\begin{proof}
As in spectral dimension $d_{s}=2$, $(\widetilde{\beta}_{i})_{i\in \mathbb{X}}$ and $(\beta_{i})_{i\in \mathbb{X}}$ are equal in distribution, same holds for $(\widetilde{u}_{i})_{i\in \mathbb{X}}$ and $(u_{i})_{i\in \mathbb{X}}$. We have
\[\mathbb{P}\left[\left\{\forall x\in[0,1), \varphi_{x}=0\right\}\right]= \mathbb{P}\left[\left\{\forall x\in[0,1), \psi_{x}=0\right\}\right]\]
and $\left\{\forall x\in[0,1), \psi_{x}=0\right\}$ is an invariant event. For $\left\{\forall x\in[0,1), \varphi_{x}>0\right\}$, we use the fact that the action of $\mathbf{G}$ is vertex transitive.
\end{proof}

Comparing this with~\cite[Proposition~3]{MR3904155}, one obtains the following dichotomy criterion:
\begin{theo}\label{th:VRJP_Dichotomy}
Consider the \(H^{2|2}\)-model on the Dyson hierarchical lattice $\mathbb{X}$ with spectral dimension $d_s=2$ and inverse temperature $\overline{W}>0$. The Vertex Reinforced Jump Process on $\mathbb{X}$ is either almost surely recurrent or almost surely transcient. Furthermore,
\begin{itemize}
    \item If every vertex is visited infinitely often, then almost surely, $\forall x\in[0,1), \varphi_{x}=0$ and $\forall i\in \mathbb{X},\ \psi_{i}=0$;
    \item If every vertex is visited finitely often, then almost surely, $\forall x\in[0,1), \varphi_{x}>0$ and $\forall i\in \mathbb{X}, \psi_{i}>0$.
\end{itemize}
\end{theo}
\begin{proof}
The proof for the martingale limit $\psi_{i},i\in \mathbb{X}$ is exactly the same as in ~\cite[Proposition~3]{MR3904155}. Now for every fixed $x_0\in[0,1)$ and $i_0\in\mathbb{X}$, the distributions of the sequences of random variables
\begin{equation*}
    (e^{u^{(n)}_{x_0})})_{n\geq 1}\quad\text{and}\quad(e^{u^{(n)}_{i_0})})_{n\geq 1}
\end{equation*}
are identical. Using the fact that the almost sure convergence to the constant $0$ is equivalent to the convergence in distribution to $0$, if $\lim\limits_{n\to\infty}\psi^{(n)}_{i_0}=0$ in the almost sure sense (thus in distribution), then so does $\lim\limits_{n\to\infty} \varphi^{(n)}_{x_0}=0$ in distribution and thus almost surely, and vice versa. This shows that almost surely the events $\left\{\forall x\in[0,1), \varphi_{x}=0\right\}$ and $\left\{\forall i\in\mathbb{X}, \psi_{i}=0\right\}$ are identical, hence the desired dichotomy criterion.
\end{proof}

\subsection{A characterization theorem: recurrence and singularity of the limiting measure}
We now show that the above zero-one law on the recurrence/transcience of the Vertex Reinforced Jump Process on the hierarchical lattice $\mathbb{X}$ characterizes the singularity of the limiting measure $\mathfrak{m}_{\mathbb{X}}$ with respect to the usual Lebesgue measure $d\lambda$ on $[0,1]$.

\begin{theo}
Recall from Theorem~\ref{th:VRJP_Dichotomy} that the Vertex Reinforced Jump Process on $\mathbb{X}$ with spectral dimension $d_s=2$ is either almost surely recurrent or almost surely transcient. Furthermore,
\begin{itemize}
    \item If every vertex is visited infinitely often, then almost surely, the limiting measure $\mathfrak{m}_{\mathbb{X}}$ is singular with respect to the Lebesgue measure $d\lambda$;
    \item If every vertex is visited finitely often, then almost surely, the limiting measure $\mathfrak{m}_{\mathbb{X}}$ has an absolutely continuous component with everywhere non-trivial density with respect to the Lebesgue measure $d\lambda$.
\end{itemize}
\end{theo}

\begin{rema}
Notice that we don't exclude here the possibility that $\mathfrak{m}_{\mathbb{X}}$ has atoms. The atomless property would follow from a classical union bound argument if we had an exponential decay estimate of type~\eqref{eq:WangZeng} with some constant $C(\overline{W},s)>2$. In constrast, the following proof does not use any quantitative estimate.
\end{rema}

\begin{proof}
Suppose now that the Vertex Reinforced Jump Process on $\mathbb{X}$ is recurrent. Then by Theorem~\ref{th:VRJP_Dichotomy}, almost surely we have that $\forall x\in[0,1), \varphi_{x}=0$. Denote by $I_n$ the nested sequence of dyadic intervals that converge to $x$. By the Lebesgue differentiation theorem~\cite[Theorem~7.10]{MR924157}, the absolute continuous part of the measure $\mathfrak{m}_{\mathbb{X}}$ has density $\lim\limits_{n\to\infty}\frac{\mathfrak{m}_n(I_n)}{\lambda(I_n)}=\varphi_{x}=0$ at Lebesgue almost every point $x$, concluding that $\mathfrak{m}_{\mathbb{X}}$ is almost surely singular in this case.

Suppose now that the Vertex Reinforced Jump Process on $\mathbb{X}$ is transient. Then by Theorem~\ref{th:VRJP_Dichotomy}, almost surely we have that $\forall x\in[0,1), \varphi_{x}>0$, and applying the Lebesgue differentiation theorem again, the Radon-Nikodym derivative (of the absolutely continuous part) $\frac{d\mathfrak{m}_{\mathbb{X}}}{d\lambda}(x)$ is exactly the limit of the exponential martingale $\varphi_{x}$ for $x\in[0,1]$. This density function is everywhere non-trivial in the transcient case by Theorem~\ref{th:VRJP_Dichotomy}, concluding the proof.
\end{proof}

\begin{rema}
The recurrence/transcience phases of the Vertex Reinforced Jump Process is a probabilistic analogy of the Anderson localization/delocalization transition. For more physical backgrounds and parallels between these notions, see~\cite{MR4680390,MR3904155,MR2953867,MR2728731}. The above theorem shows that we have an exact characterization of the recurrence/transcience of the Vertex Reinforced Jump Process on the Dyson hierarchical lattice in terms of the (pure) singularity of a natural random measure $\mathfrak{m}_{\mathbb{X}}$ on $[0,1]$ with respect to the Lebesgue measure.
\end{rema}

\bibliographystyle{alpha}

\end{document}